%% file: main.tex
\documentclass[]{interact}

\usepackage{epstopdf}
\usepackage[caption=false]{subfig}

\usepackage[numbers,sort&compress]{natbib}
\bibpunct[, ]{[}{]}{,}{n}{,}{,}

\usepackage{physics}
\usepackage{mathrsfs}
\usepackage{pdflscape}
\usepackage{mathtools}
\usepackage{algorithm}
\usepackage{algpseudocode}
\usepackage{hyperref}
\usepackage{xcolor}

\newcommand{\E}{\mathbb{E}}
\newcommand{\expvalb}[2]{\E_{#1}\left[ #2 \right]}
\newcommand{\expvalc}[3]{\E_{#1}^{#3}\left[ #2 \right]}
\newcommand{\inner}[1]{\left\langle #1 \right\rangle}

\theoremstyle{plain}
\newtheorem{thm}{Theorem}[section]
\newtheorem{prop}[thm]{Proposition}
\newtheorem{lem}[thm]{Lemma}

\theoremstyle{definition}
\newtheorem{definition}[thm]{Definition}
\newtheorem{assumption}{Assumption}

\theoremstyle{remark}
\newtheorem{rem}{Remark}

\begin{document}

\articletype{Research Article}

\title{Conjugate continuous-discrete projection filter via sparse-grid quadrature}

\author{
\name{Muhammad F. Emzir\textsuperscript{a}\thanks{Corresponding author. Email: muhammad.emzir@kfupm.edu.sa},
Zaid A. Sawlan\textsuperscript{a}, and
Sami El Ferik\textsuperscript{a}}
\affil{\textsuperscript{a}King Fahd University of Petroleum \& Minerals, Dhahran, Saudi Arabia}
}

\maketitle

\begin{abstract}
In this article, we study the continuous-discrete projection filter for exponential-family manifolds with conjugate likelihoods. We first derive the local projection error of the prediction step of the continuous-discrete projection filter. We then derive the exact Bayesian update algorithm for a class of discrete measurement processes with additive Gaussian noise. To control the stiffness of the natural parameters' ordinary differential equations (ODEs), we introduce a regularization method via projection to the Fisher information metric's eigenspace. Lastly, we apply the proposed method to approximate the filtering density of a modified Van der Pol oscillator problem and a coupled stochastic FitzHugh--Nagumo (FhN) system. The proposed projection filter shows superior performance compared to several state-of-the-art parametric continuous-discrete filtering methods.
\end{abstract}


\section{Introduction}

Optimal filter applications have become ubiquitous in the era of data abundance. Yet, except for a few known problems where analytical solutions exist, the optimal filter solution needs to be approximated. One approximation of the optimal filtering solution is the projection filter, first introduced in the early 1990s \cite{hanzon1991}. Initially formulated for continuous stochastic dynamics with a continuous measurement process, the projection filter projects the dynamics of conditional densities given by the Kushner--Stratonovich stochastic partial differential equations onto the tangent space of square-root parametric densities. The projection filter underwent a series of significant theoretical developments in the 1990s \cite{brigo1995,brigo1999}, and also recently \cite{armstrong2016,armstrong2016a,armstrong2023},\cite{gao2020a}. The applications of the projection filter were initially restricted to univariate dynamics or to Gaussian density families; see \cite{brigo1999,armstrong2023,kutschireiter2022}.  In a series of developments \cite{emzir2023a, emzir2023,emzir2024,emzir2024a,emzir2024b}, the projection filter has been re-implemented for multivariate applications via sparse-grid integration and an adaptive bijection. In this scheme, various expectations in the projection filter algorithm are computed via sparse-grid quadrature techniques, where the quadrature nodes are constantly shifted in the sample space via the adaptive bijection to focus on the high-density regions of the parametric density. This advancement paves the way for developing continuous-discrete projection filters for multivariate cases, especially using exponential-family manifolds.

Besides the projection filter, there are numerous parametric continuous-discrete filtering algorithms. Gaussian approximations have long been used \cite{jazwinski1970,arasaratnam2010,kulikov2019,knudsen2019,wang2022e,kulikov2022}, but for many nonlinear filtering problems, a single Gaussian is insufficient. Gaussian-sum filters (GSFs) \cite{alspach1972,terejanu2011,raihan2018} address this by using mixture densities, though they face challenges such as spurious modes and costly weight optimization. The particle-Gaussian-mixture (PGM) filter \cite{raihan2018} propagates particles to estimate mixand parameters via clustering, but can be computationally expensive; see also \cite{craft2025,kim2026}.
Other approaches include spectral-based methods \cite{challa2000}, Edgeworth expansion-based methods \cite{challa2000,singer2008,li2018}, and particle filters \cite{sarkka2006,xia2013}. However, particle filters require exponentially many particles with the dimension to avoid weight collapse \cite{snyder2008,beskos2017}.

The continuous-discrete projection filter is also related to Fokker--Planck solvers, including classical methods \cite{risken1996}, polynomial log-density expansions \cite{dipaola2002}, and physics-informed neural networks \cite{hu2025,chen2021}. However, neural-network-based approaches require extensive training specific to each stochastic differential equation (SDE), and their applicability to continuous-discrete filtering remains limited. In this paper, we use the sparse-grid integration to compute the log partition function \cite{bungartz2004}. Further, natural-gradient and Riemannian-optimization methods \cite{martens2020,boumal2023} motivate the numerical scheme used to approximate the Riemannian gradient in Section~\ref{sec:numerical_implementation}. For general background on Bayesian state estimation, we refer the reader to \cite{sarkka2013}.

Based on the aforementioned review, we introduce a novel continuous-discrete projection filter algorithm for a class of dynamical systems under discrete observations with Gaussian noise. In particular, the contributions of this article are the following:
\begin{enumerate}
  \item A sparse-grid quadrature based implementation of the continuous-discrete projection filter for exponential-family manifolds with conjugate likelihood.
  \item New theoretical results on projection error reduction through natural statistics augmentation (Lemma \ref{lem:norm_zero_affine} and the subsequent proposition) and an exact Bayesian update equation for the class of continuous-discrete projection filter considered (Proposition \ref{prp:conjugate_condition}), extending beyond the identity covariance case in \cite{brigo1999}.
  \item A simple regularization method via projection to the Fisher information metric's eigenspace to control the stiffness of the ODE corresponding to the natural parameters' prediction step is given in Algorithm \ref{alg:riem_grad}. The theoretical result that describes the relation between the truncation parameter in Algorithm \ref{alg:riem_grad} and the stiffness of natural parameters' ODE is given in Proposition \ref{prp:stiffness_bound_main}.
  \item Numerical experiments that compare the performance and computational efficiency of the proposed method against the Ensemble Kalman Filter (EnKF) \cite{evensen2009}, the Gaussian-Sum Filter (GSF) \cite{alspach1972}, and the Particle-Gaussian-Mixture (PGM) filter \cite{raihan2018}.
\end{enumerate}

We focus on conjugate exponential families because the approximation error then reduces to the local projection error of the Fokker--Planck equation only. The natural statistics $c = [c_1^\top, c_2^\top]^\top$ are split so that $c_1$ consists of monomials up to order $n_o$ for moment-based Fokker--Planck approximation, while $c_2$ ensures conjugacy with the measurement likelihood. This guarantees exact Bayesian updates without numerical optimization. For non-conjugate cases, Riemannian optimization \cite{emzir2025a} can be used at additional cost. The method suits a class of problems where particle filters are considered too costly to compute but Gaussian or Gaussian mixture methods are deemed insufficient. In particular, the SDEs are such that the projected Fokker--Planck equation is well-behaved and can be computed efficiently via sparse-grid quadrature, and the measurement noise is Gaussian, so the conjugate update below replaces the inner optimization needed by \cite{emzir2025a}.

\section{Problem Formulation}\label{sec:Problem_Formulation}

Consider the continuous-discrete filtering problem on the state-space model:
\begin{gather}
	\label{eqs:problem_statement}
	\begin{aligned}
		dx_t &= f(x_t) \, dt + \varrho(x_t) \, dW_t, \quad t \in [0, T],\\
		y_k &\sim p(y_k\mid x_k), \quad k = 1, 2, \ldots, K,
	\end{aligned}
\end{gather}
where $x_t \in \mathbb{R}^d$ is the state process, $y_k \in \mathbb{R}^{d_y}$ are discrete-time observations at times $t_k = k\Delta t$ for a sampling interval $\Delta t > 0$, $W_t$ is a standard Wiener process in $\mathbb{R}^{d_w}$, and $p(y_k\mid x_k)$ is the likelihood function of the measurement process.
The objective of the continuous-discrete filtering is to compute the conditional probability density function $p(x_t | \mathcal{Y}_k)$ where $\mathcal{Y}_k = \{y_1, \ldots, y_k\}$ denotes the history of observations up to time $t_k$. The conditional density evolves in two phases:
\begin{enumerate}
  \item Prediction step (between measurements): For $t \in (k \Delta t, (k+1) \Delta t)$ the conditional density satisfies the Fokker--Planck equation corresponding to the SDE in \eqref{eqs:problem_statement}.
  \item Bayesian update step: At $t= k \Delta t$, the density is updated according to Bayes' rule using the likelihood $p(y_k \mid x_k)$.
\end{enumerate}
Except for a few cases \cite{daum1984}, the filtering density cannot be computed analytically and must be approximated. The difficulties in solving the continuous-discrete filtering problem include the fact that the Fokker--Planck equation is infinite dimensional, the nonlinear dynamics introduce non-Gaussian conditional densities, and especially for multivariate systems, calculating the expectation cannot be done exactly and its approximation can be expensive to compute. In what follows, we propose a method to approximate the conditional density by following the projection filter approach \cite{brigo1998}. Specifically, we choose to work on $\mathrm{EM}(c)$, the exponential manifold (EM) of the exponential family with natural statistics vector $c$. Throughout the paper, EM denotes this exponential manifold, except when it appears in combination with PGM; in that case, PGM (EM) refers to the variant of the PGM filter that uses the Expectation--Maximisation clustering algorithm.

We then project the Fokker--Planck equation corresponding to \eqref{eqs:problem_statement} onto the tangent space of square root of $\mathrm{EM}(c)$. Further, we have carefully selected $c$ such that $\mathrm{EM}(c)$ conjugates with the likelihood $p(y_k\mid x_k)$, enabling exact Bayesian updates without optimization and expectation calculation. All expectations required for the prediction step are computed via sparse-grid quadrature in combination with the adaptive bijection introduced in \cite{emzir2023}. This reduces the computational cost while still maintaining the same level of accuracy. The measurement sampling interval $\Delta t$ is problem dependent, following the measurement devices or observation method used. For the proposed projection filter to be effective, $\Delta t$ has to be long enough that the ODE propagation can be completed before the arrival of the next measurement, but short enough that the prediction-step propagation via the projected Fokker--Planck equation when solved numerically does not leave the natural parameters' space, and that the local projection error does not accumulate too much over one inter-sample window. Precise statements for the natural parameter space and the local projection error will be developed in the next section.

\section{Mathematical Preliminaries}\label{sec:Mathematical_Preliminaries}

\subsection{Notation}\label{sec:Notation}
For an $m$-dimensional manifold $M$ with a chart $(U,\phi)$, we denote $\partial_i \coloneqq \pdv{}{\phi^i}$; i.e., for a $p\in U$, an $f \in C_p^\infty(U)$, and $r^i$ as the $i$-th coordinate of $\mathbb{R}^m$, $\left.\pdv{ }{\phi^i}\right|_p f =\left.\pdv{}{r^i}\right|_{\phi(p)} \left(f \circ \phi^{-1} \right)$. We denote the Fisher information matrix $g(\theta)$ with lower indices $g(\theta)_{ij}$, while its inverse is denoted with upper indices $g(\theta)^{ij}$.
Further, we denote the tangent space of a manifold $M$ at a point $p$ as $T_p M$ and the tangent bundle of $M$ as $TM$. For a smooth mapping $F$ between two manifolds $M$ and $N$, we denote $F_\ast$ as the differential of $F$, such that for $X_p \in T_pM$,  $F_\ast X_p \in T_{F(p)} N$ is the \emph{push-forward} of $X_p$. For a parametric density $p_{\theta}$, we denote $\E_{\theta} \left[ \cdot \right] := \E_{p_{\theta}} \left[ \cdot \right]$.
\subsection{Exponential-Family Manifolds and Projection}\label{sec:Exponential_Family_Manifold}

Let us define a class of probability densities $\mathcal{P}$ with respect to the Lebesgue measure on a fixed domain $\mathcal{X} \subseteq \mathbb{R}^d$ as $\mathcal{P} = \{p \in L^1(\mathcal{X}) : \int_\mathcal{X} p(x) \, dx = 1,\, p(x) \geq 0,\, \forall x \in \mathcal{X}\}$. In particular, let us consider the exponential family
\begin{align}
    \text{EM}(c) \coloneqq \left\{ p \in \mathcal{P} \colon p(x) = \exp\left( c(x)^\top \theta - \psi(\theta) \right) \right\},\label{eq:EM_c}
\end{align}
where $\theta \in \Theta \subset \mathbb{R}^m$ is the natural parameter, and $c\colon\mathbb{R}^d \to \mathbb{R}^m$ is a vector of natural statistics that are assumed to be linearly independent. We set $\mathcal{X}=\mathbb{R}^d$. This is consistent with the natural statistics used throughout the paper (monomials of $x$ up to order $n_o$, plus bounded statistics such as $\sin$ and $\cos$), which give a regular exponential family with $\Theta$ open in $\mathbb{R}^m$. The same construction works for state spaces with hard boundaries, provided $c$ is chosen to enforce the boundary (for example, $\log x$ for positive states). The natural parameter space $\Theta$ is defined as $ \Theta \coloneqq \{ \theta \in \mathbb{R}^m\colon \int_\mathcal{X} \exp(c(x)^\top \theta) \, dx < \infty \}.$
An exponential family is said to be regular if $\Theta$ is an open subset of $\mathbb{R}^m$. The cumulant-generating function is defined by $\psi(\theta) = \log \int_\mathcal{X} \exp(c(x)^\top \theta) \, dx$, $\theta \in \Theta$. Standard properties of regular exponential families (Theorems~2.2.1 and~2.2.5 of \cite{kass1997}) give that $\Theta$ is convex, $\psi(\theta)$ is strictly convex and infinitely differentiable on $\Theta$, and the moments and Fisher information satisfy $
    \E_\theta\left[ c_i \right] = \pdv{\psi(\theta)}{\theta_i} ,
    g_{ij}(\theta) = \pdv[2]{\psi(\theta)}{\theta_i}{\theta_j}.$
If the representation is minimal, then $g$ is positive definite.
The choice of $c$ has a strong effect on the filter's accuracy. We will make the selection precise in Section \ref{sec:Augmenting_Natural_Statistics}.

Following \cite{brigo1999}, the projection filter described here is developed on the manifold of square-root exponential family densities $\text{EM}(c)^{\frac{1}{2}} :=\{\sqrt{p_{\theta}}: p_{\theta} \in \text{EM}(c)\}$, where $\text{EM}(c)$ is given by \eqref{eq:EM_c}. As is common in information geometry \cite{amari2000}, we work with a single chart $(\text{EM}(c)^{\frac{1}{2}}, \zeta)$. Here, $\zeta : \text{EM}(c)^{\frac{1}{2}} \to \mathbb{R}^m$ with the identification $\zeta(\sqrt{p_{\theta}}) \coloneqq \theta$, where $\sqrt{p_{\theta}} \coloneqq \sqrt{p(\cdot\, ;\theta)}$. The differential of the map $\zeta$ is denoted by $\zeta_\ast$; i.e., for $X \in T_{\sqrt{p_{\theta}}}\text{EM}(c)^{\frac{1}{2}}$, $\zeta_\ast X \in T_\theta \mathbb{R}^m \cong \mathbb{R}^m$.
We equip the square-root parametric-densities manifold with the Fisher information metric, which is given by
$
\inner{\partial_i,\partial_j}_{\sqrt{p_{\theta}}}
  = \int_\mathcal{X} \frac{\partial \sqrt{p_{\theta}}}{\partial \theta_i} \frac{\partial \sqrt{p_{\theta}}}{\partial \theta_j} \, dx
  = \frac{1}{4} g(\theta)_{ij}
$.
With this metric, the square-root parametric-densities manifold becomes a Riemannian manifold, where notions of inner product, distance, and projection are well-defined. In particular, for any $v \in L^2(\mathcal{X})$, the projection onto $T_{\sqrt{p_{\theta}}}\text{EM}(c)^{\frac{1}{2}}$ is given by \cite{brigo1999}:
\begin{equation}
  \Pi_{(\sqrt{p_{\theta}})} v = \sum_{i=1}^m \sum_{j=1}^m  4 g(\theta)^{ij}\inner{v,\partial_j}_{\sqrt{p_{\theta}}} \partial_i. \label{eq:proj_EM_c_half}
\end{equation}

\section{Projection Filter Theory}\label{sec:Projection_Filter_Theory}

\subsection{Local Projection Error Analysis}\label{sec:Local_projection_error}

Recall from Section \ref{sec:Problem_Formulation} the state-space model \eqref{eqs:problem_statement}.
We use the same set of assumptions as mentioned in \cite[\textsection 3]{brigo2009}:
\begin{assumption}
  $x_0$ has an almost surely positive density $p_0$ with respect to the Lebesgue measure, \label{asm:x_0}
\end{assumption}
\begin{assumption}\label{asm:f_and_varrho}
  $f \in C^1(\mathcal{X}), \varrho \varrho^\top \in C^2(\mathcal{X})$.
\end{assumption}
\begin{assumption}
  there exists $K>0$, such that, $\forall x, 2 x^\top f(x) + \norm{\varrho(x)\varrho(x)^\top} \leq K (1 + \norm{x}^2)$, \label{asm:x_top_f}
\end{assumption}
\begin{assumption}
  the law of $x_t$ corresponding to \eqref{eqs:problem_statement} is absolutely continuous, and its density is $C^2(\mathcal{X})$ and satisfies the Fokker--Planck equation \label{asm:law_of_x}
\end{assumption}
Assumptions \ref{asm:x_0}--\ref{asm:x_top_f} ensure the existence of a unique solution to the SDE, while Assumption \ref{asm:law_of_x} guarantees that, for all $t$ between sampling times, the corresponding filtering density $p_t$ in \eqref{eqs:problem_statement} exists and belongs to $C^2(\mathcal{X})$.

The continuous-discrete projection filter is implemented in two steps. In the first part, the dynamics of the square-root density are projected onto the tangent space $T_{\sqrt{p_{\theta}}}\text{EM}(c)^{\frac{1}{2}}$. The second step is the Bayesian update step, where the measurement $y_k$ is incorporated to correct $p_{\theta_{k}^-}$ based on the likelihood density $p(y_k\mid x_k) = \exp(-\ell_y)$. The posterior density is given by

\begin{equation}
  q = p_{\theta_k^-} \exp(-\ell_y-Z(\theta_k^-)), \quad\quad
  Z(\theta_k^-) = \log[\E_{\theta_k^-}\left[ \exp(-\ell_y )\right]] \label{eq:posterior}.
\end{equation}

In this section, we will focus on the first step of the continuous-discrete projection filter. Let $t \in ((k-1) \Delta t, k \Delta t)$. The dynamics of the square-root density $\sqrt{p_t}$ related to the state process \eqref{eqs:problem_statement} follows the differential equation
\begin{equation}
    \frac{d \sqrt{p_t}}{d t} = \frac{1}{2 \sqrt{p_t}} \mathcal{L}^*(p_t). \label{eq:sqrt_p_t_dynamics}
\end{equation}
Using \eqref{eq:sqrt_p_t_dynamics}, \eqref{eq:proj_EM_c_half}, and \cite[Lemma 2.1]{brigo1999}, the projection of $ \frac{d \sqrt{p_{\theta_t}}}{d t}$ onto $T_{\sqrt{p_{\theta_t}}}\text{EM}(c)^{\frac{1}{2}}$ is given by
\begin{equation}
  \Pi_{(\sqrt{p_{\theta_t}})} \frac{ d \sqrt{p_{\theta_t}}}{d t}
    = \sum_{i=1}^m\sum_{j=1}^m g(\theta)^{ij} \E_{\theta_t}[\mathcal{L}(c_j)] \partial_i. \label{eq:projected_sqrt_p_dynamics}
\end{equation}
In this equation, $\mathcal{L}$ is the backward Kolmogorov diffusion operator and $\mathcal{L}^*$ is its adjoint. For brevity, we denote the natural statistics expectation as $\eta(\theta) \coloneqq \E_{\theta}[c]$.
The difference between $\sqrt{p_t}$ and $\sqrt{p_{\theta_t}}$ can be formulated as follows. Suppose $t_0 = 0$, $dt > 0$, and initially $p_0 = p_{\theta_0}$. Then
$
    \sqrt{p_{dt}} - \sqrt{p_{\theta_{dt}}} = \frac{dt}{2 \sqrt{p_{\theta_0}}} \left( \mathcal{L}^* (p_{\theta_0}) - p_{\theta_0}(c - \eta(\theta_0))^\top g(\theta_0)^{-1} \E_{\theta_0}[\mathcal{L}(c)] \right)  \\
    + \mathcal{O}(dt^2).
$

Taking $dt \to 0$, after some manipulations, we have
\begin{equation}
    \frac{4}{dt^2} \norm{\sqrt{p_{dt}} - \sqrt{p_{\theta_{dt}}}}_{L^2(\mathcal{X})}^2 = \E_{\theta_0}\left[ \left( \frac{\mathcal{L}^* (p_{\theta_0})}{p_{\theta_0}} \right)^2 \right] - \E_{\theta_0}[\mathcal{L}(c)]^\top g(\theta_0)^{-1}\E_{\theta_0}[\mathcal{L}(c)].\label{eq:norm_instant_projection_error}
  \end{equation}

The difference between $\sqrt{p_t}$ and $\sqrt{p_{\theta_t}}$ can now be written in an integral form, where, by using \eqref{eq:norm_instant_projection_error}, we have the following upper bound
$
    \norm{\sqrt{p_t} - \sqrt{p_{\theta_t}}}_{L^2(\mathcal{X})} \leq \frac{1}{2} \int_0^t \left( E_1(\sqrt{p_{\theta_\tau}}) + E_2(\sqrt{p_\tau}, \sqrt{p_{\theta_\tau}}) \right) d\tau,
$
with,
\begin{subequations}
  \begin{align}
    E_1(\sqrt{p_{\theta}};c) \coloneqq& \left[ \E_{\theta}\left[ \left( \frac{\mathcal{L}^*(p_{\theta})}{p_{\theta}} \right)^2 \right] - \E_{\theta}[\mathcal{L}(c)]^\top g(\theta)^{-1}\E_{\theta}[\mathcal{L}(c)] \right]^{1/2}, \label{eq:proj_Error_type_1}\\
    E_2(\sqrt{p}, \sqrt{p_{\theta}}) \coloneqq& \norm{\frac{1}{2 \sqrt{p}} \mathcal{L}^*(p) - \frac{1}{2 \sqrt{p_{\theta}}} \mathcal{L}^*(p_{\theta})}_{L^2(\mathcal{X})}. \label{eq:proj_Error_type_2}
\end{align}
\end{subequations}
The term $E_1$ above can be considered as the discrepancy between the square-root projected density and the square-root density at $t + dt$ given that they are equivalent at time $t$. The accumulation of this error is related to the total projection error when $E_2$ is neglected. While the second error term $E_2$ might not be computable, the first projection error term $E_1(\sqrt{p_{\theta_t}};c)$ can be computed numerically under sufficient conditions that $f \in C^1(\mathcal{X})$, $\varrho, c \in C^2(\mathcal{X})$ with bounded derivatives, and $\E_{\theta} \left[ \left( \frac{\mathcal{L}^*(p_{\theta})}{p_{\theta}} \right)^2 \right] < \infty$, uniformly on $\Theta$.
\subsection{Reducing Projection Error via Natural Statistics Augmentation}\label{sec:Augmenting_Natural_Statistics}

We will show that expanding natural statistics by augmenting linearly independent elements into existing natural statistics always reduces the local projection error. Mathematically, consider $c^\top = [c_1^\top, c_2^\top]$, where $m = m_1 + m_2$, $c_1: \mathbb{R}^d \to \mathbb{R}^{m_1}$ is the first set of natural statistics, and $c_2: \mathbb{R}^d \to \mathbb{R}^{m_2}$ is another set of natural statistics to be selected.
Choose $\theta^\top = [\theta_1^\top, 0^\top]$. In this case, $p_{\theta} = \exp(c^\top \theta - \psi(\theta)) = \exp(c_1^\top \theta_1 - \psi(\theta_1))$.
We can think of $c$ as an extension of $c_1$ with additional statistics $c_2$. From the projection filter's perspective, in addition to reducing the projection error $E_1$, the reason for extending the natural statistics is that we can choose $c_1$ to be the set of monomials with maximum order $n_o$ such that $\text{EM}(c_1)$ can be considered a good approximation to the solution of the Fokker--Planck equation up to moments of order $n_o$. Then the extension $c_2$ is chosen to ensure that the augmented natural statistics $c = [c_1^\top, c_2^\top]^\top$ form an exponential-family manifold $\text{EM}(c)$ such that the likelihood function from the discrete observation process $y_k \sim p(y_k | x_k)$ is conjugate to $\text{EM}(c)$. For example, in the Van der Pol example in Section \ref{sec:Numerical_Simulations}, we select $n_o=4$ to capture mean, covariance, skewness, and kurtosis. Then, we extend it with $c_2$ based on $h(x_k)$ to achieve the conjugacy requirement.
This way, we ensure that the only sources of approximation errors for the whole continuous-discrete filtering steps (up to numerical accuracy) arise from the projection error of the Fokker--Planck equation. We will make the conjugacy requirement precise in Section \ref{sec:Exact_Bayesian_Update}. Before that,
 the question that we would like to answer here is whether at $\theta$, the local projection error using the expanded natural statistics $c$, $E_1(\sqrt{p_{\theta}};c)$ is always less than using only $c_1$, $E_1(\sqrt{p_{\theta_1}};c_1)$. If the expanded natural statistics $c$ still have linearly independent elements, then according to \eqref{eq:proj_Error_type_1} this question can be answered if
\begin{equation}
    \E_{\theta}[\mathcal{L}(c)]^\top g(\theta)^{-1}\E_{\theta}[\mathcal{L}(c)] > \E_{\theta}[\mathcal{L}(c_1)]^\top g_{(11)} (\theta)^{-1}\E_{\theta}[\mathcal{L}(c_1)], \label{eq:strict_L_g_inv_L_inequality}
\end{equation}
where,
\begin{equation*}
    g(\theta) = \begin{bmatrix}
        g_{(11)} (\theta) & g_{(12)}(\theta) \\
        g_{(12)}(\theta)^\top & g_{(22)} (\theta)
    \end{bmatrix}, \quad g_{(ij)} (\theta) \coloneqq \E_{\theta} [\tilde{c}_i \tilde{c}_j^\top],
\end{equation*}
and $\tilde{c}_i = c_i - \E_{\theta} [c_i]$. To proceed further, we need the following assumption and lemma.
\begin{assumption}\label{asm:x_and_x_square_in_c1}
The natural statistics vector $c_1$ has linearly independent elements that include $x_i$, $x_i x_j$ for $i,j = 1, \dots, d$, such that $\Theta_1 \coloneqq \{\theta_1 \in \mathbb{R}^{m_1}: \int \exp(c_1^\top \theta_1) dx < \infty \}$ is open.
\end{assumption}

\begin{lem}\label{lem:norm_zero_affine}
Consider an $m$-dimensional exponential-family manifold $\text{EM}(c) = \{ p_{\theta} = \exp(c^\top \theta - \psi(\theta)) : \theta \in \Theta \}$ with $c^\top = [c_1^\top, c_2^\top]$, where $c_i : \mathbb{R}^d \to \mathbb{R}^{m_i}$ for $i = 1, 2$, $m = m_1 + m_2$, and $c$ is second-order continuously differentiable with respect to $x$, where $c_1$ satisfies Assumption \ref{asm:x_and_x_square_in_c1}.
Then \eqref{eq:norm_zero_requirement} is satisfied for all $\theta \in \Theta$ if and only if $c_2$ is an affine function of $c_1$.
\begin{equation}
  \E_{\theta} [\mathcal{L}(c_2)] = g_{(12)}(\theta)^\top g_{(11)} (\theta)^{-1} \E_{\theta} [\mathcal{L}(c_1)]. \label{eq:norm_zero_requirement}
\end{equation}

\end{lem}

The proof is given in Appendix \ref{app:proof_lemma}.

Consider the validity of the inequality \eqref{eq:strict_L_g_inv_L_inequality} for any $\theta \in \Theta$. In the following proposition, we show that when $g(\theta)$ is invertible, then the inequality \eqref{eq:strict_L_g_inv_L_inequality} is always satisfied.
\begin{prop}\label{prp:strict_inequality}
  If the family $\text{EM}(c)$ is minimal, then \eqref{eq:strict_L_g_inv_L_inequality} holds.
\end{prop}
Intuitively, the equality in \eqref{eq:strict_L_g_inv_L_inequality} would require $c_2$ to be an affine function of $c_1$ (by Lemma~\ref{lem:norm_zero_affine}), which is ruled out as long as the family $\text{EM}(c)$ is minimal; the Schur-complement decomposition in Appendix~\ref{app:proof_schur} makes this precise.
The full proof is in Appendix \ref{app:proof_schur}. Motivated by the above explanation, let us define precisely the necessary conditions for a set of statistics $c_2$ to be considered an extension of $c_1$.
\begin{definition}\label{def:extension_c}
  A function $c_2: \mathcal{X} \to \mathbb{R}^{m_2}$ is an extension of $c_1$ with respect to $\text{EM}(c_1)$ if $c_2$ is affinely independent of $c_1$ and there exists a non-empty open neighborhood of zero $\Theta_2 \subset \mathbb{R}^{m_2}$ such that $\Theta \subset \Theta_1 \otimes \Theta_2 \subset \mathbb{R}^{m_1 + m_2}$ is defined by
  \begin{equation}
      \Theta = \left\{ \theta_1 \in \Theta_1, \theta_2 \in \mathbb{R}^{m_2} : \int_{\mathcal{X}} \exp(c_1^\top \theta_1 + c_2^\top \theta_2) \, dx < \infty \right\} \label{eq:Theta_extension}
  \end{equation}
\end{definition}
The definition above ensures that the extended natural statistics $c^\top = [c_1^\top, c_2^\top]$ forms a minimal exponential-family manifold $\text{EM}(c)$ such that $\text{EM}(c_1) \subset \text{EM}(c)$. Notice that, for the case of a univariate problem, if $c_1 = [x, x^2]$, then $c_2 = [x^3]$ is \textbf{not} an extension of $c_1$ since, although it is linearly independent of all entries of $c_1$, there is no open set on $\mathbb{R}$ such that $\Theta$ in \eqref{eq:Theta_extension} exists. However, if we choose $c_2 = [x^3, x^4]$, then $c_2$ is an extension of $c_1$.

\subsection{Exact Bayesian Update for Conjugate Likelihoods}\label{sec:Exact_Bayesian_Update}

In the case that the likelihood function $\exp(-\ell_y)$ is conjugate to the prior, the posterior $q$ given in \eqref{eq:posterior} also belongs to $\text{EM}(c)$, and the Bayesian update can be performed exactly. The following lemma gives the precise condition for the conjugacy.
  \begin{lem}[\cite{emzir2025a}]\label{lem:conjugate}
  Let $\theta_0\in \Theta$ and $q = p_{\theta_0}\exp(-\ell_y - Z(\theta_0))$, where the support of $\exp(-\ell_y)$ is also $\mathcal{X}$, and $p_{\theta_0} \in \text{EM}(c)$. If there exists $\theta_\ell \in \mathbb{R}^m$ where the negative log likelihood can be written as $\ell_y = c^\top \theta_\ell$, and $\exp(Z(\theta_0))<\infty$, then $q = p_{\theta_\ast}$ where $\theta_\ast = \theta_0 - \theta_\ell \in \Theta$.
  \end{lem}

  In what follows, we describe an explicit Bayesian update algorithm for a class of measurement processes with additive Gaussian noise by giving the expression of $\theta_\ell$ as a function of the measurement $y_k$. The case where $R = I$ has been derived in \cite[\textsection 6.2]{brigo1999}.
  \begin{prop}\label{prp:conjugate_condition}
    Let the discrete observation process be given by $y_k = h(x_k) + v_k$, where $v_k \sim \mathcal{N}(0, R)$. If there exists a subset $c_h: \mathcal{X} \to \mathbb{R}^{m_h}$ of natural statistics $c$ such that $h = H^\top c_h$, and the entries of its Kronecker product $c_h \otimes c_h$ are also entries of $c$, then
    $
        \theta_\ell = -K_1 y_k + \theta_{(\ell, 0)}, \label{eq:theta_ell}
    $
    where $K_1 = T_1^\top H R^{-1}$ and $\theta_{(\ell, 0)} = \frac{1}{2} T_2^\top \text{vec}(H R^{-1} H^\top)$, and the matrices $T_1, T_2$ are defined such that:
    \begin{align}
      c_h &= T_1 c, & c_h \otimes c_h &= T_2 c. \label{eq:T_1_T_2}
  \end{align}
\end{prop}

The proof is given in Appendix \ref{app:proof_conjugate}.
\begin{rem}
In light of the proposition above, the measurement vector $y_k$ enters the Bayesian update step linearly. Furthermore, the shift $\theta_{(\ell, 0)}$ can be precomputed before the simulation since it is a constant. Lastly, by Lemma \ref{lem:conjugate}, we have the Bayesian update given by the following simple equation:
\begin{equation}
  \theta_k = \theta_k^- - \theta_{(\ell, 0)} + K_1 y_k. \label{eq:ExactBayesianUpdate}
\end{equation}
\end{rem}
\begin{rem}
  To use Proposition \ref{prp:conjugate_condition}, if $y_k=h(x_k) + v_k$ and $c_1$ has been selected (e.g., monomials of $x$ up to order $n_o$), then one can select $c_h$ as the remaining basis such that elements of $h$ can be expanded in terms of $c_1$ and $c_h$. Then one finally needs to check if $c_2$ can be selected to be an augmentation of $c_h$ and $c_h\otimes c_h$, and whether $c_2$ extends $c_1$ in the sense of Definition \ref{def:extension_c}. If so, then Proposition \ref{prp:conjugate_condition} holds. For example, in Section \ref{sec:Numerical_Simulations} for the Van der Pol simulation, where $h(x) = [\sin(x_1), \sin(x_2)]^\top$ and $c_1=\{x^{\boldsymbol{i}} : |\boldsymbol{i}| \leq 4\}$, we can choose $c_h = h$.
\end{rem}
\begin{rem}\label{rem:non_gaussian_noise}
Proposition \ref{prp:conjugate_condition} assumes Gaussian measurement noise $v_k \sim \mathcal{N}(0,R)$. When the noise is not Gaussian, two workarounds are available. (i) One can check whether the conditions in Lemma \ref{lem:conjugate} are satisfied; if so, $\theta_k = \theta_k^- - \theta_\ell$ for some $\theta_\ell$ as described in Lemma \ref{lem:conjugate}.   (ii) For general noise densities, one can fall back to the Riemannian-gradient update of \cite{emzir2025a}, which casts the Bayesian step as an inner $\alpha$-R\'enyi divergence minimization problem on the exponential-family manifold at the cost of an inner optimization per measurement.
\end{rem}
\section{Numerical Implementation}\label{sec:numerical_implementation}

To calculate the cumulant-generating function $\psi(\theta)$ and expectations with respect to $p_\theta$ in \eqref{eq:projected_sqrt_p_dynamics}, we use the sparse-grid quadrature combined with the adaptive Gaussian-based bijection from \cite{emzir2023} to adaptively cover the high-density region of $p_\theta$. The $d$-dimensional sparse-grid quadrature is constructed from a unidimensional grid via what is known as Smolyak's construction \cite{smolyak1963}. The unidimensional grids like the Gauss--Patterson and Gauss--Kronrod grids have $N(1,l) = 2^{l+1}-1$ nodes, where $l$ is the sparse-grid level, and $N(d,l)$ is the total number of quadrature nodes. An important feature of Smolyak's construction is that for a $d$ dimensional sparse-grid, the sparse-grid nodes have only $\mathcal{O}(2^l l^{d-1})$ quadrature nodes instead of $N(1,l)^d$ nodes as is the case with the tensor product of unidimensional grids \cite{novak1996,gerstner1998}. Hence, the sparse-grid quadratures scale modestly with the dimension. On the other hand, for many particle filter methods, the number of particles often scales exponentially with the dimension of the problem to avoid particle weight collapse \cite{snyder2008,beskos2017}. In terms of approximation accuracy, for any function belonging to the Sobolev space $H^r([0,1]^d)$, the sparse-grid's approximation error is bounded by $\mathcal{O}(2^{-l r}l^{(d-1)(r+1)})$, where $r$ is a constant related to smoothness \cite{gerstner1998}. Due to the smoothness of the projection density, a small number of nodes can be used to achieve an accurate approximation.

A brief description of adaptive bijection combined with sparse-grid integration is as follows. The adaptive bijection, denoted as $\phi_\xi:\mathbb{R}^d \to \mathbb{R}^d$, where the parameter of the bijection is given by $\xi=\left(\mu, \Sigma\right)$, is defined as
$
  \phi_\xi(\tilde{y}) = \mu + \sqrt{2} L \tilde{y}, \quad \tilde{y} \in \mathbb{R}^d.
$
In this bijection, $\mu = \expvalb{\theta}{x}$, $\Sigma = \expvalb{\theta}{(x - \mu)(x - \mu)^\top}$, $L$ is the Cholesky decomposition of $\Sigma$, and $\tilde{y}_i := \erf^{-1}(\tilde{x}_i)$, where $\tilde{x}_i \in \mathcal{D}_c$ are the original quadrature nodes in the hypercube domain $(-1,1)^d$. Using this bijection, for a function $\varphi : \mathbb{R}^d \to \mathbb{R}$, we define $d$-dimensional quadrature with $N$ nodes as
$
  Q^d_N \left[ \varphi \right] \coloneqq \int_{\mathbb{R}^d} \varphi (\tilde{y}) \, d\tilde{y} \approx \sum_{i=1}^N w_{s,i} \, \varphi(\tilde{y}_i),
$
where $w_{s,i} = \left(\frac{1}{2}\sqrt{\pi}\right)^d \exp(\norm{\tilde{y}_i}^2) w_i$. The cumulant-generating function is approximated by $\psi^N(\theta)$ as
$
  \psi^N(\theta) \coloneqq \log\left( Q^d_N\left[\exp(c^\top(\phi_\xi)\theta) 2^{\frac{d}{2}}\det(L)\right]\right).
$
Furthermore, for a function $\varphi(x):\mathcal{X} \to \mathbb{R}$, we approximate its expectation with respect to $p_\theta$ as
$
  \expvalc{\theta}{\varphi}{N} \coloneqq Q^{d}_N \left[ \varphi(\phi_\xi) \exp(c^\top(\phi_\xi)\theta - \psi^N(\theta)) 2^{\frac{d}{2}}\det(L)\right].
$
By Assumption \ref{asm:x_and_x_square_in_c1}, elements of $\expvalb{\theta_t}{x}$ and $\expvalb{\theta_t}{xx^\top}$ are included in $\eta(\theta_t) = \E_{\theta_t}[c]$. This implies that there exist $T_\mu \in \mathbb{R}^{d \times m}$ and a linear map $\Phi_\Sigma: \mathbb{R}^{m} \to \mathbb{R}^{d \times d}$ such that $\mu_t = T_\mu \eta(\theta_t)$ and $\expvalb{\theta_t}{xx^\top} = \Phi_\Sigma(\eta(\theta_t))$. Therefore, the prediction step of the projection filter can be implemented by solving the following differential equation for $t \in ((k-1)\Delta t, k \Delta t)$:
\begin{gather}
  \begin{aligned}
    \dv{\eta_t}{t} &= \expvalc{\theta_t}{\mathcal{L}(c)}{N},&
    \dv{\theta_t}{t} &= \left( g^N(\theta_t) \right)^{-1} \dv{\eta_t}{t}, \\
    \dv{\mu_t}{t} &= T_\mu \dv{\eta_t}{t},&
    \dv{\Sigma_t}{t} &= \Phi_\Sigma\left(\dv{\eta_t}{t}\right) - \dv{\mu_t}{t}^\top \mu_t - \dv{\mu_t}{t} \mu_t^\top.
  \end{aligned}\label{eqs:ode_prediction_step}
\end{gather}

\begin{algorithm}
  \caption{Single-Step Projection Filter}\label{alg:projection_filter}
  \begin{algorithmic}[1]
  \Require $\theta_{k-1},y_k,\xi_{k-1}$
  \State $\theta_k^- , \xi_k^- \gets$ \Call{Solve ODEs \eqref{eqs:ode_prediction_step}}{$\theta_{k-1}; \xi_{k-1}, \Delta t$}
  \State $\theta_k \gets$ \Call{Bayesian Update \eqref{eq:ExactBayesianUpdate}}{$\theta_k^-, y_k$}
  \State $\eta(\theta_k) \gets \pdv{\psi^N(\theta_k)}{\theta}$ \Comment{via automatic differentiation}
  \State $\mu_k \gets T_\mu \eta(\theta_k)$
  \State $\Sigma_k \gets \Phi_\Sigma(\eta(\theta_k)) - \mu_k \mu_k^\top$
  \State $\xi_k \gets (\mu_k, \Sigma_k)$ \Comment{Update bijection parameter}
  \State \Return $\theta_k, \xi_k$
  \end{algorithmic}
\end{algorithm}

The single-step implementation of the proposed continuous-discrete projection filter is presented in Algorithm \ref{alg:projection_filter}.

Based on our numerical experiments (for details, see Section \ref{sec:Numerical_Simulations}), we noticed that Algorithm \ref{alg:projection_filter} can be unstable for some random measurement record samples. In particular, the approximated Fisher information matrix $g^N(\theta_t)$ can have eigenvalues less than or equal to zero, which leads to numerical instability when it is used to propagate ODEs \eqref{eqs:ode_prediction_step}.
In this work, we modify the Riemannian gradient in the ODE \eqref{eqs:ode_prediction_step} by adjusting the Fisher information matrix according to Algorithm \ref{alg:riem_grad} that truncates the eigenvalues of Fisher information matrix's inverse and limit $\tfrac{d\theta}{dt}$. Since $g^N(\theta)$ is symmetric, $g^N(\theta) = \sum_i \lambda_i v_i v_i^\top$ with eigenvalues $\lambda_1 \geq \lambda_2 \geq \cdots \geq \lambda_m$ and orthonormal eigenvectors $v_i$. Algorithm \ref{alg:riem_grad} ensures that the natural gradient is computed only on the eigen-subspace associated with eigenvalues larger than $\epsilon$. As is clear, in Algorithm \ref{alg:riem_grad}, when all eigenvalues are above $\epsilon$ and $\norm{w}_2 < M$, $w = g^N(\theta)^{-1} v$. This sets the algorithm apart from the Tikhonov regularization approach where $g^N(\theta)+\lambda I$ is used to compute the natural gradient. First of all, this not only introduces unnecessary bias in the natural gradient computation, but it is also unclear how to select $\lambda$ optimally, since the eigenvalues of $g^N(\theta)$ can vary significantly during the simulation. Furthermore, the parameter $\epsilon$ in Algorithm \ref{alg:riem_grad} has a clear interpretation as the stiffness control for the natural parameter dynamics as the following proposition shows.

\begin{algorithm}
\caption{Eigenvalue Truncated Riemannian Gradient}
\label{alg:riem_grad}
\begin{algorithmic}[1]
\Require Fisher matrix $g \in \mathbb{R}^{m \times m}$, Euclidean gradient vector $v \in \mathbb{R}^m$, threshold $\epsilon \ge 0$, maximum norm $M > 0$
\State $g \gets \frac{1}{2}(g + g^\top)$ \Comment{Ensure the matrix is symmetric}
\State $(\Lambda, V) \gets \mathrm{eigen}(g)$ \hfill \Comment{where $\Lambda = \mathrm{diag}(\lambda_1, \dots, \lambda_m)$}
\State $s \gets 0 \in \mathbb{R}^m$ \Comment{Set initial reciprocal vector to zero}
\For{$i = 1$ \textbf{to} $m$}
    \If{$\lambda_i > \epsilon$}
        \State $s_i \gets 1 / \lambda_i$ \Comment{Only use $\lambda_i>\epsilon$}
    \EndIf
\EndFor
\State $w \gets V S_\epsilon V^\top v$ \Comment{Here, $S_\epsilon = \mathrm{diag}(s)$}
\State $w \gets w \cdot \frac{\min(\|w\|_2, M)}{\|w\|_2}$ \Comment{Apply norm clipping}
\Return $w$
\end{algorithmic}
\end{algorithm}

  \begin{prop}\label{prp:stiffness_bound_main}
    Consider the natural parameter's dynamics given in \eqref{eqs:ode_prediction_step}, and let $\omega(\theta) = (g^N(\theta))^{-1}\mathbb{E}^N_\theta[\mathcal{L}(c)]$. Let $g^N(\theta) = V \Lambda V^\top$ with eigenvalues $\lambda_1 \geq \cdots \geq \lambda_m$. Let $S_\epsilon$ be the diagonal matrix constructed as per Algorithm \ref{alg:riem_grad}. Assume that the following conditions are satisfied:
    \begin{enumerate}
    \item $\max_{k \in \{1,\ldots,m\}}\|\mathrm{Cov}_\theta^N(\mathcal{L}(c),c_k)\|_2 \leq \bar{C}$,
    \item $\max_{k \in \{1,\ldots,m\}} \|\frac{\partial g^N(\theta)}{\partial \theta_k}\|_2 \leq \bar{T}$,
    \item $\|\mathbb{E}^N_\theta[\mathcal{L}(c)]\|_2 \leq \bar{V}$,
    \item The minimum eigenvalue gap among active eigenvalues satisfies $\delta \coloneqq \min_{i,j \in \mathcal{A}, i \neq j} |\lambda_i - \lambda_j| > 0$.
  \end{enumerate}
  Let $\omega_\epsilon (\theta) = G_\epsilon(\theta)\E_\theta^N[\mathcal{L}(c)]$, where $G_\epsilon (\theta) \coloneqq VS_\epsilon V^\top$. Then the Jacobian of $\omega_\epsilon$ satisfies
  \begin{equation}
    \norm{\dfrac{\partial \omega_\epsilon}{\partial \theta}}_2 \leq \frac{\bar{C}}{\epsilon} + \left( \frac{1}{\epsilon} + \frac{2}{\delta} \right) \frac{\bar{T}\bar{V}}{\epsilon}.\label{eq:df_epsilon_d_theta_bound}
  \end{equation}
  \end{prop}
The proof has two parts. The sensitivity of $\E^N_\theta[\mathcal{L}(c)]$ in $\theta$ is handled with the standard exponential-family identity $\partial_{\theta_k} \E_\theta[\varphi] = \mathrm{Cov}_\theta(\varphi, c_k)$. The sensitivity of $G_\epsilon(\theta)$ is handled by Weyl's inequality (for the eigenvalues) and the Davis--Kahan theorem (for the eigenvectors), giving a $1/\delta$ scaling in the spectral gap. The full proof is in Appendix \ref{app:proof_stiffness}.
Finally, we collect some remarks that address some aspects of Algorithm \ref{alg:riem_grad} and Proposition \ref{prp:stiffness_bound_main}.
  \begin{rem}\label{rem:epsilon_effect_and_M}
    Proposition \ref{prp:stiffness_bound_main} shows how the parameter $\epsilon$ relates to the stiffness of \eqref{eqs:ode_prediction_step} via \eqref{eq:df_epsilon_d_theta_bound}. For near-singular $g^N(\theta)$, the second term dominates stiffness. A larger $\epsilon$ value lowers the bound \eqref{eq:df_epsilon_d_theta_bound} and could allow adaptive ODE solvers to take larger time steps while also discarding gradient information along low-curvature directions. For practical use, $\epsilon$ should be small enough that the truncated Fisher matrix is well-conditioned, and $M$ should be chosen so that a single ODE step is unlikely to leave $\Theta$. In our experiments, the values of $\epsilon$ and $M$ are tuned manually by inspecting a small number of trial runs.
  \end{rem}
  \begin{rem}
    Let $\theta_t \in \Theta$ and let $w$ be the output of Algorithm \ref{alg:riem_grad}, where $\epsilon \geq0$ and $M>0$. Define the maximum radius from $\theta$ to the boundary of $\Theta$ as $r(\theta) \coloneqq \sup \left\{ r> 0 : \mathcal{B}(\theta, r) \subset \Theta \right\}$, where $\mathcal{B}(\theta, r)$ is an open ball with center at $\theta$ and radius $r$. If the step size satisfies $\delta t \leq \tfrac{r(\theta_t)}{M}$, then $\theta_{t+\delta t}\coloneqq \theta_t + \delta t w \in \Theta$. In practice, $r(\theta_t)$ varies along the trajectory and is generally unknown. Therefore, Algorithm \ref{alg:riem_grad} with fixed $M$ does not guarantee $\theta_t \in \Theta, \forall t >0$. Notice that \eqref{eqs:ode_prediction_step} is also known as the natural gradient in the deep-learning literature \cite{martens2020}. As the natural gradient is usually used for optimization purposes, either the analysis of whether the parameter can cross outside the parameter set is completely neglected or the parameter space is assumed to be $\mathbb{R}^n$ and hence parameter constraints can be handled via reparametrization. In our case, the parameter space $\Theta$ for regular exponential families is in general a proper subset of $\mathbb{R}^m$.
  \end{rem}
  \begin{rem}
    To completely force the solution $\theta_t \in \Theta$ for all $t$, one would need either an ODE solver that abides by the manifold structure of $\Theta$ or a projection applied at each ODE step to project $\theta_{t}$ onto $\Theta$. However, unless the exponential family is Gaussian, there is no trivial way to do so, since $\theta$ corresponds to a symmetric tensor of dimension $d$ of order greater than two; see \cite{emzir2024b,schmudgen2017}.
  \end{rem}

The cost of each filter step has three contributions: sparse-grid quadrature, the Fisher matrix and its eigen-decomposition, and the Bayesian update. The dominant scaling with dimension is set by the sparse-grid node count. As we will see below in Section \ref{sec:Numerical_Simulations}, for the four-dimensional FitzHugh--Nagumo example that uses level 6 Gauss--Kronrod quadrature, a total of $186{,}345$ nodes is used; extending the same construction to substantially larger $d$ is not feasible on a single workstation. For larger $d$, possible alternatives include using the adaptive sparse-grid quadrature \cite[Section 5]{gerstner1998}, or formulating a higher accuracy bijection that can capture higher moments rather than merely first and second moments as was proposed in \cite{emzir2023}. The overall view of the proposed method is depicted in Figure \ref{fig:flowchart_methodology}.

\begin{figure}[H]
\centering
\includegraphics[width=\linewidth]{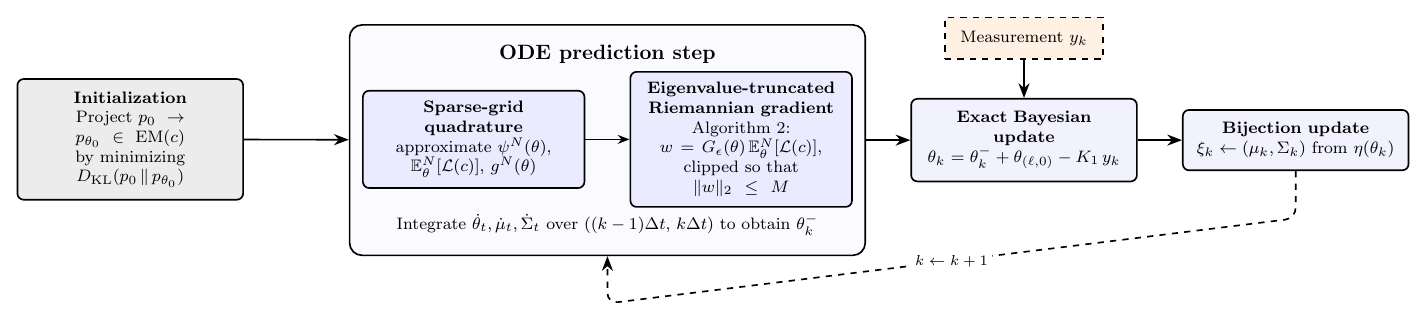}
\caption{Methodology flowchart of the proposed conjugate continuous-discrete projection filter.}
\label{fig:flowchart_methodology}
\end{figure}

\section{Numerical Simulations}\label{sec:Numerical_Simulations}
\subsection{Two Dimensional System}\label{sec:vdp_example}
Consider the modified stochastic Van der Pol dynamics with nonlinear measurement, defined as follows:
\begin{subequations}
  \begin{align}
    d \begin{bmatrix} x_{1,t} \\ x_{2,t} \end{bmatrix} &= \begin{bmatrix} x_{2,t} \\ \mu_s (1 - x_{1,t}^2) x_{2,t} - x_{1,t} \end{bmatrix} dt + \begin{bmatrix} 0 \\ \sigma_w \end{bmatrix} dW_t, \label{eq:vdp_state_dynamic} \\
    y_k &=  \begin{bmatrix} \sin(x_{1,k}) & \sin(x_{2,k}) \end{bmatrix}^\top + \sigma_v v_k. \label{eq:vdp_measurement}
\end{align}
\end{subequations}
In these equations, $x_k \coloneqq x_{k \Delta t},~y_k \coloneqq y_{k \Delta t},~ \mu_s = 0.5, ~\sigma_v = 1, \sigma_w = 2$, $~\Delta t = 1$, and $v_k \sim \mathcal{N}(0,I)$. The initial density is set to be a mixture Gaussian density $p_0 = 0.5 \mathcal{N}([1 , -1 ]^\top, I) + 0.5 \mathcal{N}([-1 , 1 ]^\top, I)$. In addition, we choose $c_1 = \{x^{\boldsymbol{i}} : |\boldsymbol{i}| \leq 4\}$. The use of monomials of up to order $4$ is designed to capture not only the mean and the covariance, but also the skewness and kurtosis of the empirical distribution. We augment $c_1$ with $c_2 = [\sin(x_1), \sin(x_2), \sin(x_1)\sin(x_2), \sin(x_1)^2,\sin(x_2)^2]^\top$ such that the natural statistics $c^\top = [c_1^\top,c_2^\top]$ satisfy the requirements of Proposition \ref{prp:conjugate_condition}. To propagate the projected Fokker--Planck equation, we use the Gauss--Kronrod sparse-grid with level 8. The total number of quadrature nodes for this sparse-grid scheme is $20,\!833$. The ODEs in Algorithm \ref{alg:projection_filter} are solved using the Tsitouras 5 scheme \cite{tsitouras2011} via DiffraX package \cite{kidger2022}. For this, we set relative tolerance to be $10^{-3}$ and absolute tolerance to be $10^{-6}$. The ODE solver time step is tuned adaptively via a PID controller with default settings. The parameters $(\epsilon,M)$ in Algorithm \ref{alg:riem_grad} were set to be varied, first we set them to $(-\infty,\infty)$, which means no regularization at all. We will refer to the simulation result using this regularizer setting as $\mathrm{Proj}$. Then we also tried $(0,\infty)$, which means $g^N(\theta)$ is projected onto the non-negative eigenspace, and finally we set them to $(1\times 10^{-5}, 10^2)$. These two settings are referred to as $\mathrm{Proj\text{-}0}$ and $\mathrm{Proj\text{-}B}$, respectively.
\subsection{Four Dimensional System}\label{sec:fhn_example}
In this section, we consider the application of the proposed projection filter algorithm to a four-dimensional stochastic system. The system we consider corresponds to a coupled FitzHugh--Nagumo (FhN) stochastic system \cite{quininao2020}; see also \cite{voss2004,smith2019}. The deterministic version of this model is central in computational neuroscience \cite{fitzhugh1955,nagumo1962}, and is normally used to analyze neural network activity with electrical coupling in a simplified framework. The FhN SDEs and the corresponding measurement process are given as follows
\begin{subequations}
  \begin{align}
    \begin{bmatrix} d x_{1,t} \\ d x_{2,t} \\ d x_{3,t}  \\ d x_{4,t} \end{bmatrix} &=
    \begin{bmatrix} x_{1,t} - \dfrac{x_{1,t}^3}{3} - x_{2,t} + i_1 + \epsilon_{1,2} (x_{3,t} - x_{1,t}) \\
      \dfrac{x_{1,t} + a - b x_{2,t}}{\tau} \\
      x_{3,t} - \dfrac{x_{3,t}^3}{3} - x_{4,t} + i_2 + \epsilon_{2,1} (x_{1,t} - x_{3,t})\\
      \dfrac{x_{3,t} + a - b x_{4,t}}{\tau}
    \end{bmatrix} dt +  \sigma_w dW_t, \label{eq:fhn_state_dynamic} \\
    y_k &=  h(x_k) + \sigma_v v_k.
\end{align}
\end{subequations}
In this equation, $x_1$ and $x_3$ are the membrane potentials of the first and second nerves, respectively. The state $x_2$ and $x_4$ are the recovery variables, and $i_1$ and $i_2$ are the external currents, which are normally set between 0 and $\frac{1}{2}$. The parameters of this dynamics are the threshold offset $a = 0.7$, the recovery strength $b=0.8$, time scale ratio $\tau=12.5$, the relaxation coefficients $\epsilon_{1,2}=\epsilon_{2,1}=0.1$, the noise strength for process and measurement (respectively) $\sigma_w=1,\sigma_v=2$, and the external currents $i_1=0.25, i_2=0.5$.
For this simulation, we set the measurement function as
$
 h(x) = [x_{1},x_{2},x_{3},x_{4}]^\top.
$
The measurement is collected every $\Delta t = 0.25$ s. The initial density is set to be a mixture Gaussian density $p_0 = 0.5 \mathcal{N}([1,1,1,1]^\top, I) + 0.5 \mathcal{N}(-[1,1,1,1]^\top, I)$. Further, we choose $c_1 = \{x^{\boldsymbol{i}} : |\boldsymbol{i}| \leq 4\}$. We notice that $h$ is already in the span of $c_1$, and hence technically we do not need to expand $c_1$ with extra statistics to satisfy Proposition \ref{prp:conjugate_condition}. However, we deliberately add a high-order statistic $c_2 = \sum_{i=1}^d x_i^6$ to demonstrate the feasibility of using high-order monomials beyond order four. Notice that in this particular simulation, since the measurement function is linear, we expect that the ensemble Kalman filter will perform better than in the previous simulation. To propagate the projected Fokker--Planck equation, we use the Gauss--Kronrod sparse-grid with level 6. The total number of quadrature nodes for this sparse-grid scheme is $186,\!345$. Similar ODE solver settings from the previous section are used again here, except that we fixed $\epsilon=10^{-3}$ and $M=10$ in Algorithm \ref{alg:riem_grad}.
\subsection{Metrics}\label{sec:metrics}
Let the approximated density from method $\mathbf{a}$ at time $k$ be denoted $p_{(\mathbf{a},k)}$, and let $\E_{(\mathbf{a},k)} [\phi]$ denote the corresponding expectation of a statistic $\phi$.
We use the following metrics:
\begin{enumerate}
    \item \emph{Hellinger Distance:}
    $H(p_k, p_{(a,k)}) = \sqrt{\frac{1}{2} \int_{\mathcal{D}} (\sqrt{p_k} - \sqrt{p_{(a,k)}})^2 \, dx}$,
    where $\mathcal{D}$ is the region within 3 Mahalanobis distances from the bijection Gaussian \cite{emzir2023a}, capped at $1$.

    \item \emph{Sliced Wasserstein-1 Distance} \cite{bonneel2015,kolouri2019}:
    $\mathrm{SW}_1(p_k, p_{(\mathbf{a},k)}) \approx \frac{1}{N_{w}} \sum_{i=1}^{N_{w}} \mathrm{W}_1(\Pi_{i} p_k, \Pi_{i}  p_{(\mathbf{a},k)})$,
    where $\Pi_i$ are random one-dimensional projections and $\mathrm{W}_1$ is the one-dimensional Wasserstein-1 distance. For the FitzHugh--Nagumo example, we replace the Hellinger distance with $\mathrm{SW}_1$ as the former becomes computationally demanding in four dimensions.

    \item \emph{Natural Statistics' Mean Squared-Error (nMSE):}
    $\mathrm{nMSE}(\mathbf{a},k) = \frac{1}{N_s} \sum_{i=1}^{N_s} \|c(x_k) - \E_{(\mathbf{a},k)} [c]\|_2^2$,
    where $N_s$ is the number of Monte Carlo runs. Unlike standard MSE, nMSE accounts for higher moments through the natural statistics $c$.

    \item \emph{Cross Entropy:}
    $\mathrm{CE} (p_k, p_{(\mathbf{a},k)}) \approx -\frac{1}{N} \sum_{i=1}^N \log(p_{(\mathbf{a},k)}(x_k^{(i)}))$.
\end{enumerate}

\subsection{Benchmark Methods}\label{sec:benchmark_methods}
We employ the continuous-discrete particle filter with $2.4 \times 10^7$ samples as the ground truth, using the Euler--Heun scheme ($\delta t=2.5 \times 10^{-2}$) and systematic resampling \cite{chopin2020}. We compare the projection filter against the Ensemble Kalman Filter (EnKF) \cite{evensen2009}, the Gaussian-Sum Filter (GSF) \cite{alspach1972}, a Sigma-Point GSF (SP-GSF) using order-5 Gauss--Hermite quadrature, the Homotopic Gaussian Mixture Filter (HGMF) obtained by replacing the per-mixand Bayesian update of SP-GSF with the homotopic Gaussian update of \cite{craft2025} (the same Gauss--Hermite sigma points are re-used to evaluate the expectations in the homotopic ODE), and the Particle-Gaussian-Mixture (PGM) filter \cite{raihan2018} with both K-means and EM clustering variants. The EnKF uses the same number of samples and SDE solver as the particle filter.

All Gaussian-mixture methods are initialized with $N_m = 50$ mixands, where the first two correspond to $p_0$ with modified weights $0.5/(1 + \epsilon_m (N_m - 2))$, $\epsilon_m = 10^{-2}$, and the remaining mixands have weights $\epsilon_m/(1 + \epsilon_m (N_m - 2))$, with means sampled from $\mathcal{N}(0, I)$ and covariances $I$. The PGMs' sample count is set to ten times the number of projection filter quadrature nodes. The total number of Gaussian-mixture parameters is $150$ (Van der Pol) and $700$ (FitzHugh--Nagumo).

In initializing the projection filter, since the initial density does not belong to $\text{EM}(c)$, we approximate the empirical density $p_0$ with $p_{\theta_0}$, where $\theta_0$ is selected to minimize the Kullback--Leibler divergence $D_{\text{KL}}(p_0 \parallel p_{\theta_0})$. This is achieved by first computing $\expvalb{p_0}{c}$ via the particle samples, and then finding $\theta_0$ using the optimization problem \cite{amari2000},
$\theta_0 = \arg \max_{\theta \in \Theta} \left(\theta^\top \expvalb{p_0}{c} - \psi(\theta)\right)$,
which we implement using Riemannian gradient descent. Due to the fact that $p_0 \notin \mathrm{EM}(c)$, the projection density $p_{\theta_0}$ resembles the initial density $p_0$ with some noticeable deformations, while the initial parametric densities for GSF, SP-GSF, PGM (K-means), and PGM (EM) are almost identical to $p_0$.
For example, for the Van der Pol example in Section \ref{sec:vdp_example}, the resulting $p_{\theta_0}$ and the corresponding initial Gaussian-mixture densities are depicted in Figure \ref{fig:density_comparison_med_time_index_0}.
  \begin{figure}[H]
  \centering
  \includegraphics[width=\textwidth]{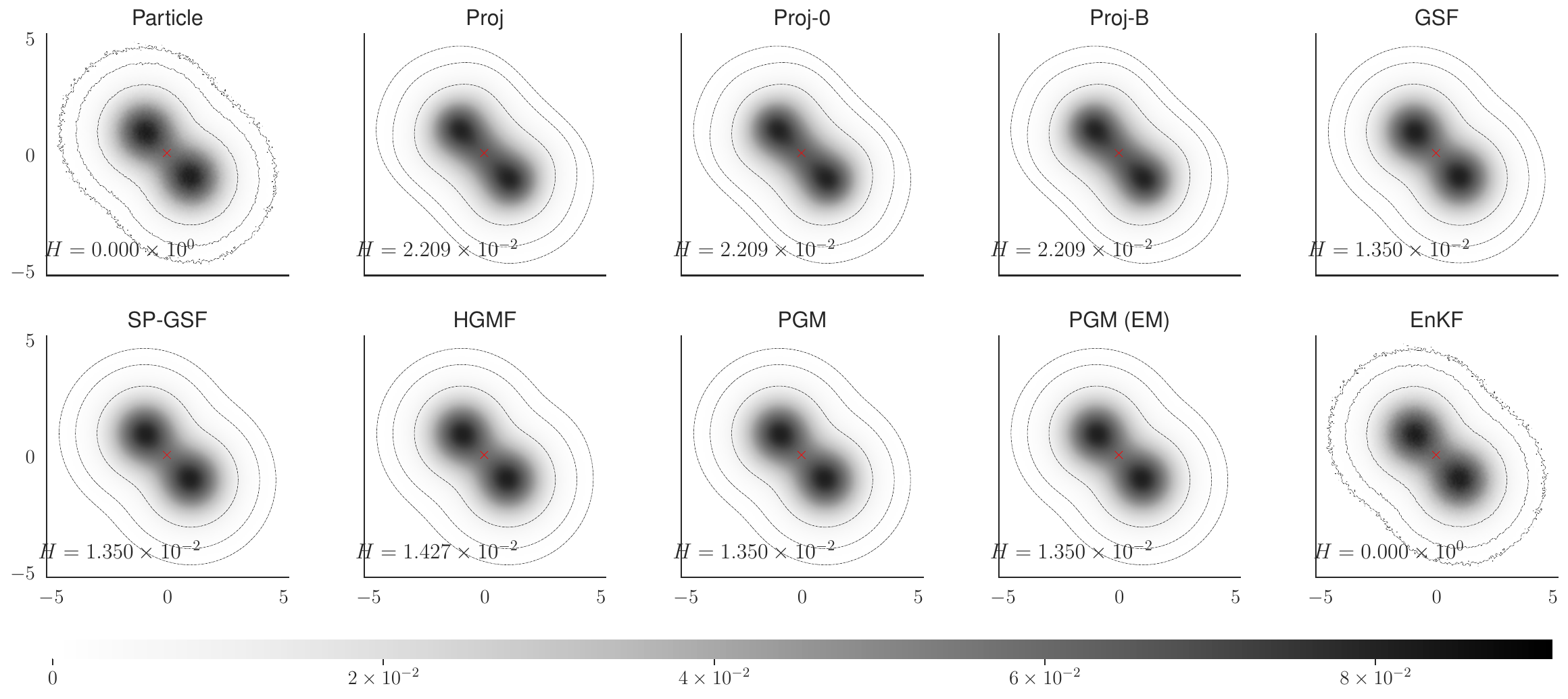}
  \caption{The initial densities $p_0$, the initial projection densities $p_{\theta_0}$ (same for Proj, Proj-0, and Proj-B), and the initial Gaussian-mixture densities for GSF, SP-GSF, HGMF, PGM (K-means), and PGM (EM). In these plots, the $\times$ sign indicates the location of the real state value. The dash-dotted lines represent the equipotential lines of the densities. To make the comparison easier, we have added contour lines (solid black) and we indicate the Hellinger distance to the empirical density at the bottom of every subplot.}
  \label{fig:density_comparison_med_time_index_0}
\end{figure}
\subsection{Simulation Results}\label{subsec:simulation_results}

\subsubsection{Two Dimensional System}

\begin{figure}[H]
  \centering
  \includegraphics[width=\textwidth]{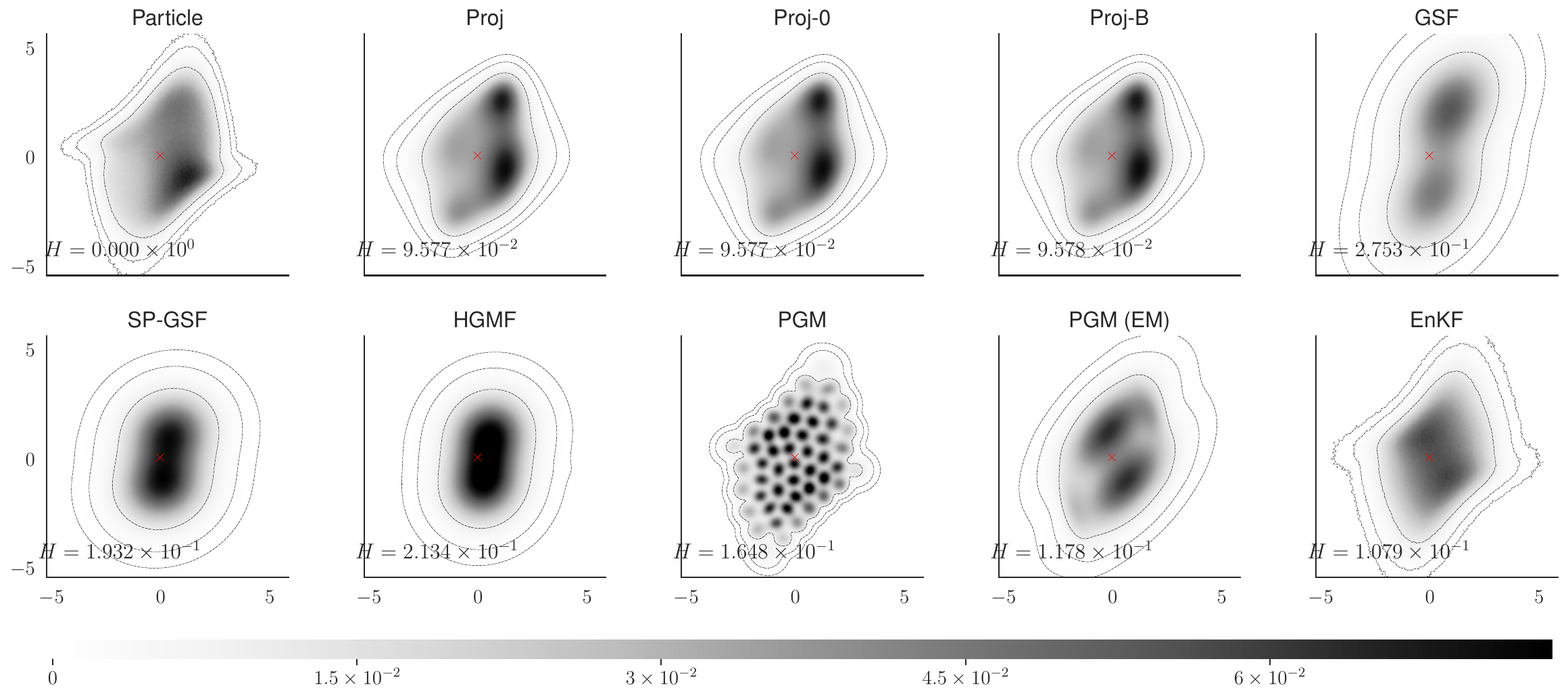}
  \caption{Comparison of the posterior empirical densities from particle filter samples and the approximated posterior densities at $k=1$. No noticeable differences are observed among the projection filters' densities (Proj, Proj-0, and Proj-B).}
  \label{fig:density_comparison_med_time_index_1}
\end{figure}
\begin{figure}[H]
  \centering
  \includegraphics[width=\textwidth]{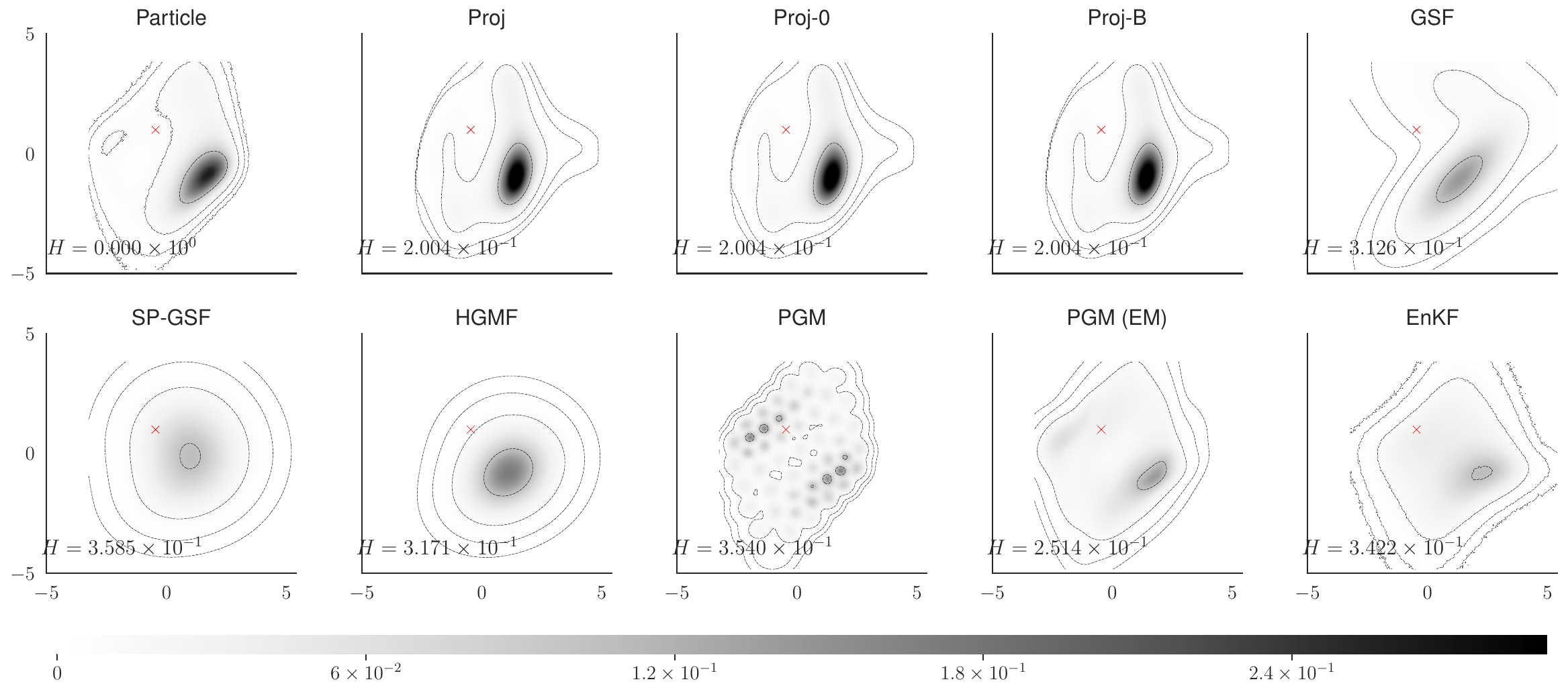}
  \caption{Similar to Figure \ref{fig:density_comparison_med_time_index_1}, but at $k=4$.}
  \label{fig:density_comparison_med_time_index_4}
\end{figure}
  For the Van der Pol filtering problem, the completion rates are $91\%$ (Proj), $92\%$ (Proj-0), and $99\%$ (Proj-B), confirming that Algorithm \ref{alg:riem_grad} improves robustness. No noticeable performance deterioration is observed from regularization, as it is only active when eigenvalues of $g^N(\theta)$ fall below the threshold.

In Figures \ref{fig:density_comparison_med_time_index_1} and \ref{fig:density_comparison_med_time_index_4}, we compare posterior densities from the median simulation run. The projection filters' densities closely resemble the empirical posterior, while GSF and SP-GSF remain near bi-Gaussian, PGM (K-means) exhibits spurious peaks, and EnKF densities do not converge to the conditional density \cite{bishop2023}.

\begin{figure}[H]
  \centering
  \includegraphics[width=\textwidth]{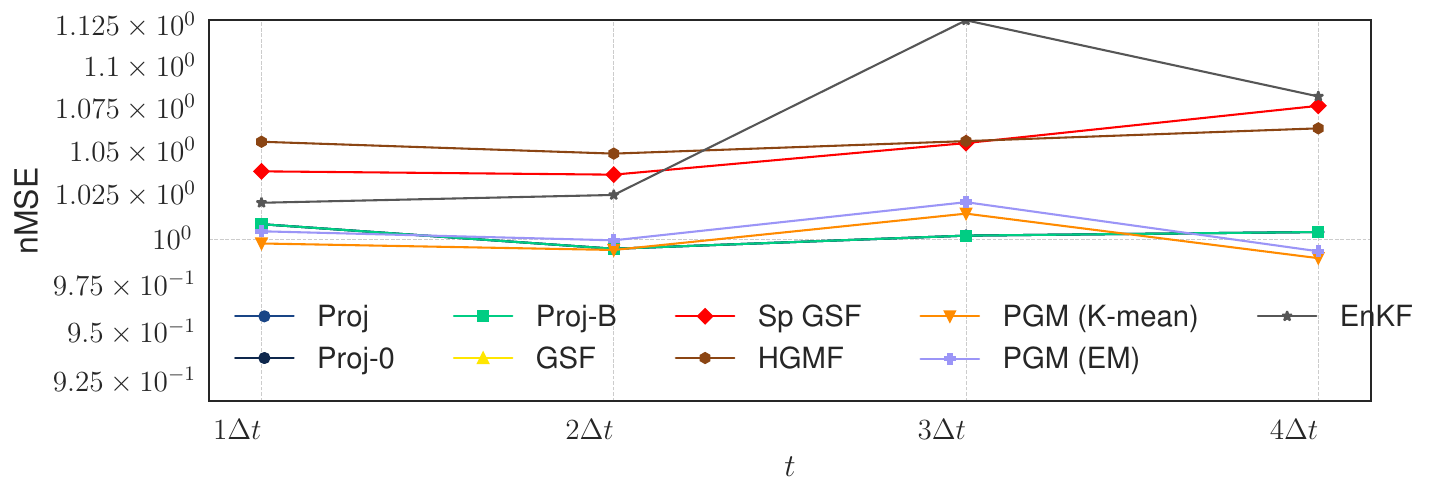}
  \caption{Comparison of $\text{nMSE}$ for different methods. The values are normalized against the $\text{nMSE}$ of the particle filter. Proj, Proj-0, and Proj-B have an overlapping $\mathrm{nMSE}$ plots, while GSF $\mathrm{nMSE}$ is more than twice higher than the particle filter hence not shown here.}
  \label{fig:natural_statistics_mse}
\end{figure}
\begin{figure}[H]
  \centering
  \includegraphics[width=\textwidth]{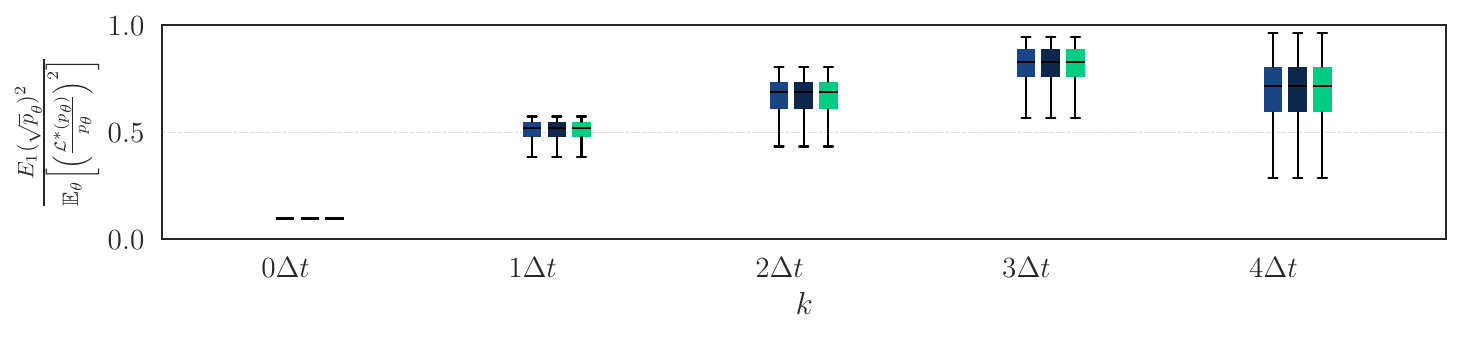}
  \caption{Quartile plots of the ratio $E_1(\sqrt{p_\theta})^2/\expvalb{\theta}{\left(\frac{\mathcal{L}^\ast (p_\theta)}{p_\theta}\right)^2}$.}
  \label{fig:local_projection_error}
\end{figure}
\begin{figure}[H]
  \centering
  \includegraphics[width=\textwidth]{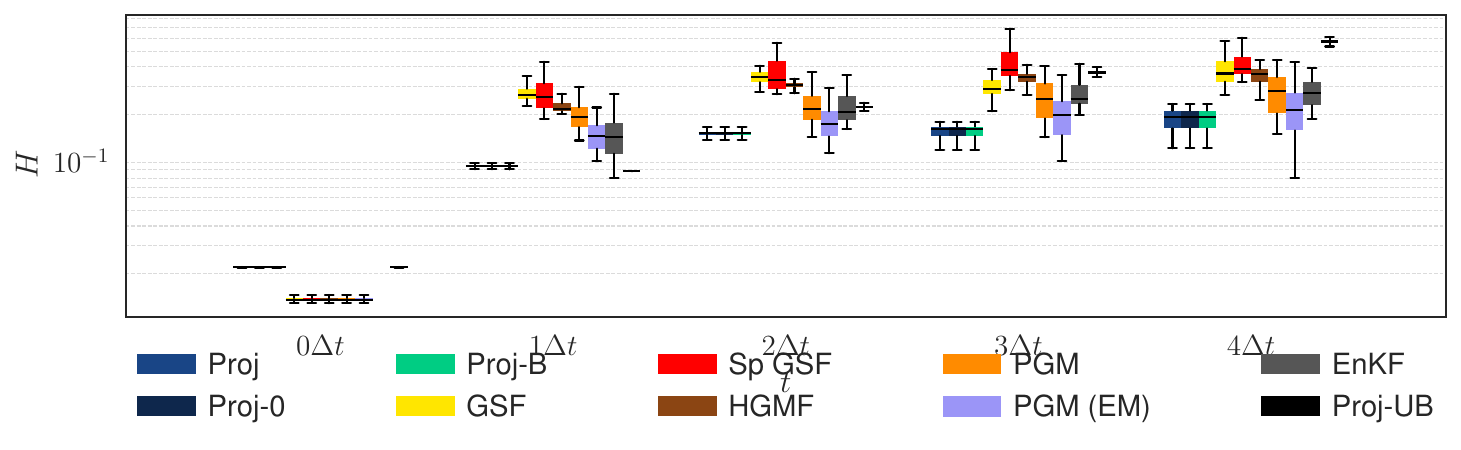}
  \caption{Quartile plots of the Hellinger distances from the approximated posterior densities to the empirical posterior density for different methods.}
  \label{fig:Hell_dist_to_particle_vdp}
\end{figure}
\begin{figure}[H]
  \centering
  \includegraphics[width=\textwidth]{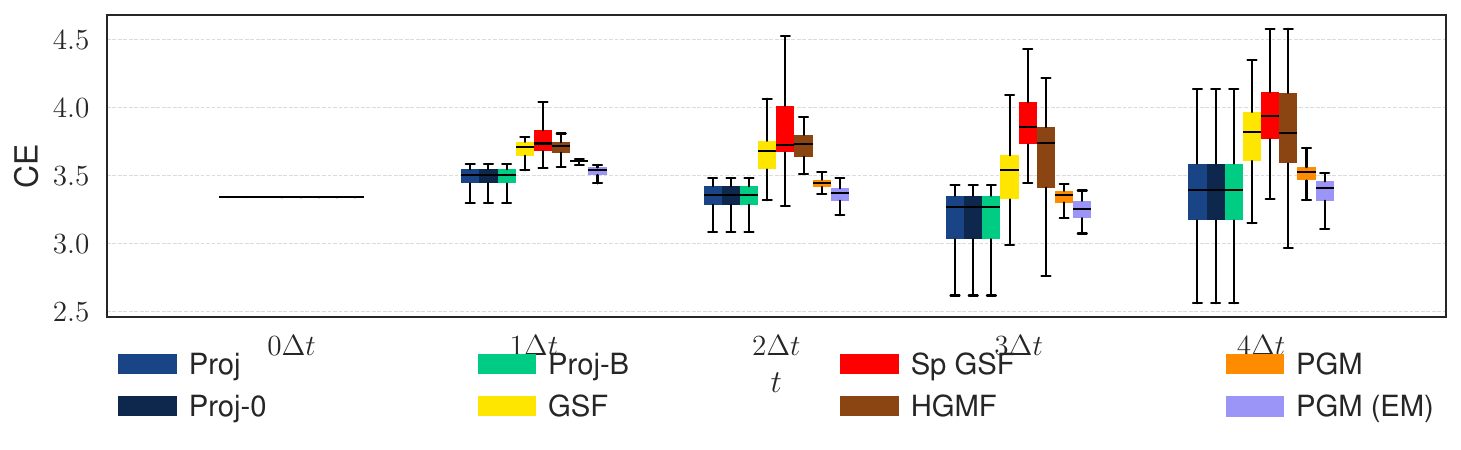}
  \caption{Quartile plots of the cross entropies from the approximated posterior densities to the empirical posterior density for different methods.}
  \label{fig:cross_entropy_vdp}
\end{figure}

In Figure \ref{fig:local_projection_error}, we observe a steady increase in the local projection error ratio over time. Using $E_1$ values, we estimate Hellinger distances via $\frac{\partial H}{\partial t} \approx \frac{1}{2\sqrt{2}} E_1(\sqrt{p_{\theta_t}})$, shown as Proj-UB in Figure \ref{fig:Hell_dist_to_particle_vdp}. Although the projection filters start with higher Hellinger distances, from $k=1$ onward they achieve the lowest values among all methods. The same trend holds for cross entropy and nMSE (Figures \ref{fig:cross_entropy_vdp} and \ref{fig:natural_statistics_mse}), with the projection filter's nMSE nearly matching the particle filter. These results demonstrate that the projection filter outperforms all competing methods despite having fewer parameters (19 vs.\ 150) and quadrature points. As shown in Table \ref{tab:result-vdp}, the projection filters are also computationally efficient, ranking fourth through sixth fastest after GSF, SP-GSF, and HGMF. The projection filter retains the smallest values on all three accuracy metrics.

\input{vdp_auto_generated_table}

\subsubsection{Four Dimensional System}
\begin{figure}
  \centering
  \includegraphics[width=1\linewidth]{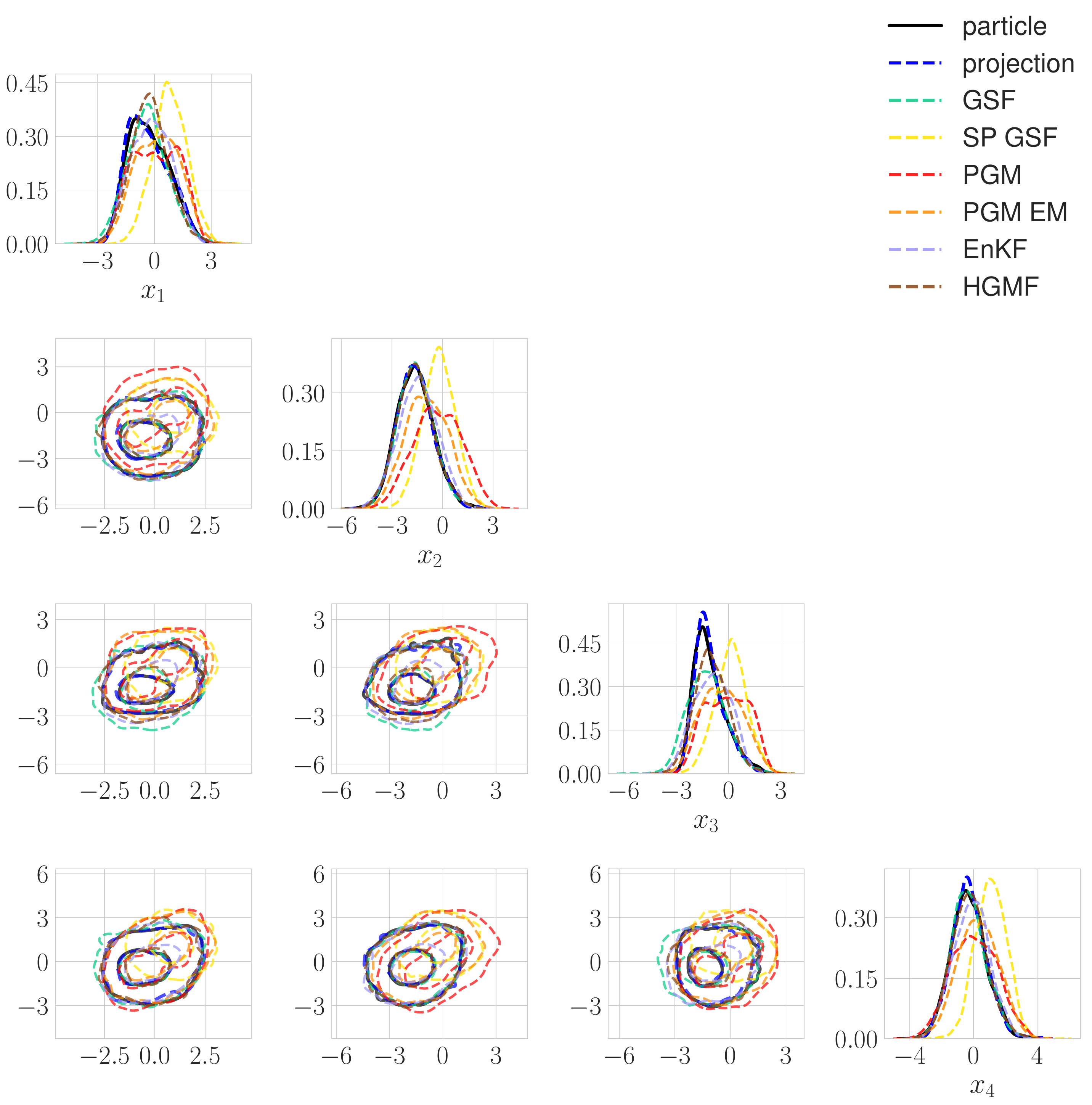}
  \caption{Comparison of the projection density at $t=\Delta t$, the empirical density, and the density of GSF, SP-GSF, PGM, PGM (EM), and EnKF (color coded). Notice that the projection density and the empirical density are indistinguishable in this time instance, while the competing methods struggle to approximate the posterior density.}
  \label{fig:Posterior_KDE_Comparison_at_t1_seed_1}
\end{figure}

\begin{figure}
  \centering
  \includegraphics[width=1\linewidth]{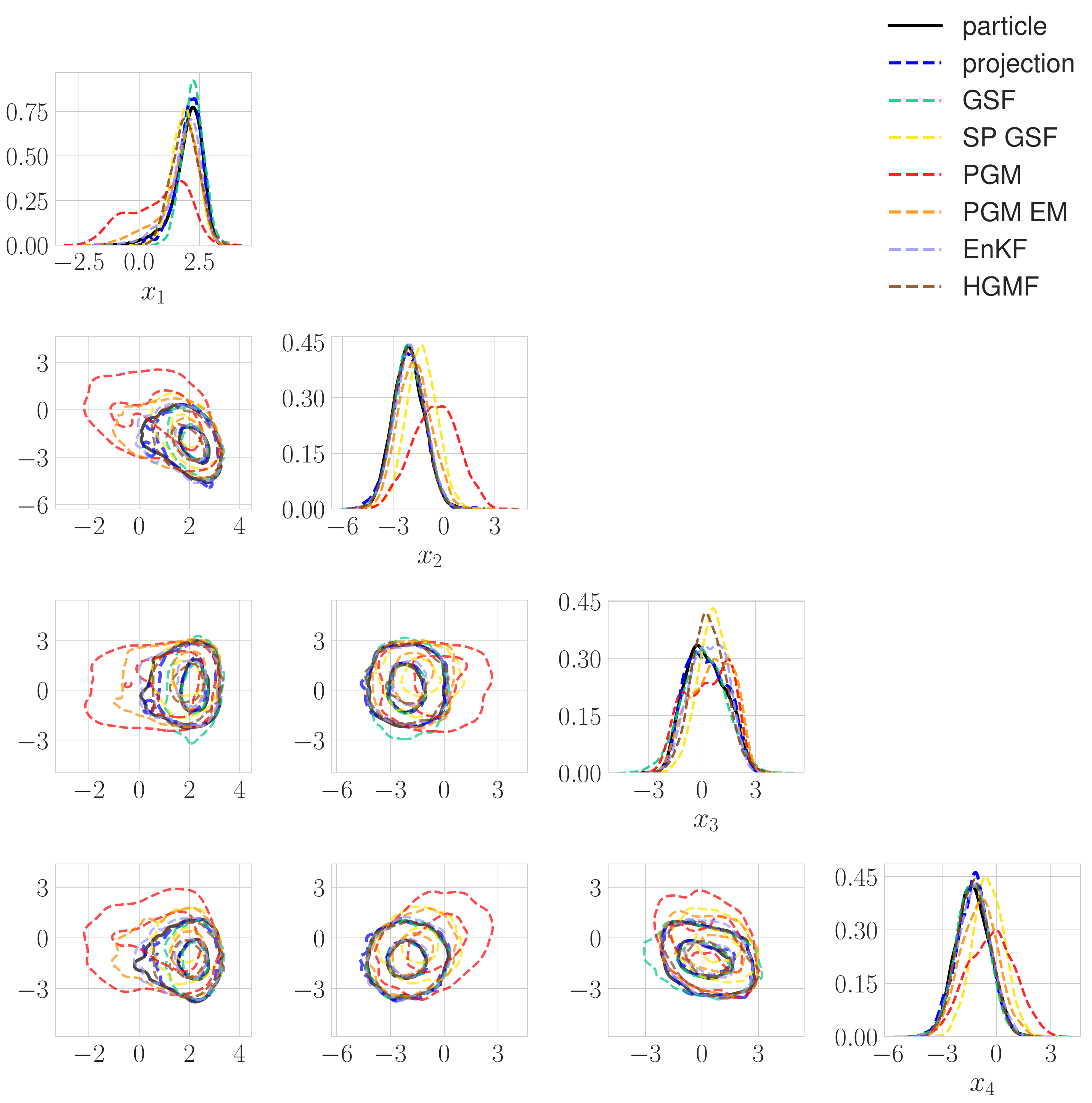}
  \caption{Similar to Figure \ref{fig:Posterior_KDE_Comparison_at_t1_seed_1}, for $t=4 \Delta t$.}
  \label{fig:Posterior_KDE_Comparison_at_t4_seed_1}
\end{figure}

For the FitzHugh--Nagumo (FhN) filtering problem, $98\%$ of simulations completed. As shown in Figures \ref{fig:sw1_dist_comparison_fhn} and \ref{fig:cross_entropy_fhn}, the projection filter achieves the lowest $\mathrm{SW}_1$ distances and cross entropy from $k=1$ onward. The marginalized density comparisons in Figures \ref{fig:Posterior_KDE_Comparison_at_t1_seed_1} and \ref{fig:Posterior_KDE_Comparison_at_t4_seed_1} show the projection filter density nearly overlapping with the empirical posterior, while competing methods exhibit noticeable deformations. The nMSE remains within $1.07$ times the particle filter's value; See Figure \ref{fig:natural_statistics_mse_fhn}.
Overall, the projection filter outperforms all competing methods with fewer parameters (70 vs.\ 700) and is the fourth-fastest algorithm (Table \ref{tab:result-fhn}).
For this example, HGMF performs better than PGM, PGM (EM), or EnKF. We attribute this to the short measurement interval such that the predictive update via sigma points is accurate enough in this regime. Finally, Tables~\ref{tab:result-vdp} and~\ref{tab:result-fhn} also show the trade-off between accuracy and computation. The projection filter uses more FLOPS per step than GSF, SP-GSF, and HGMF, while tracking the empirical posterior more closely on Hellinger and SW1 metrics, and propagating far fewer parameters (19 vs.~150 for Van der Pol; 70 vs.~700 for FhN). The proposed method is therefore generally both cheaper and better than PGM methods, and thus suitable for applications where a higher accuracy parametric multimodal posterior is necessary, and simple Gaussian-mixture-based methods (e.g.\ GSF, SP-GSF, HGMF, PGM, PGM (EM)) are deemed less accurate, in exchange for extra computational costs.

\input{fhn_auto_generated_table}

\begin{figure}
  \centering
  \includegraphics[width=0.95\linewidth]{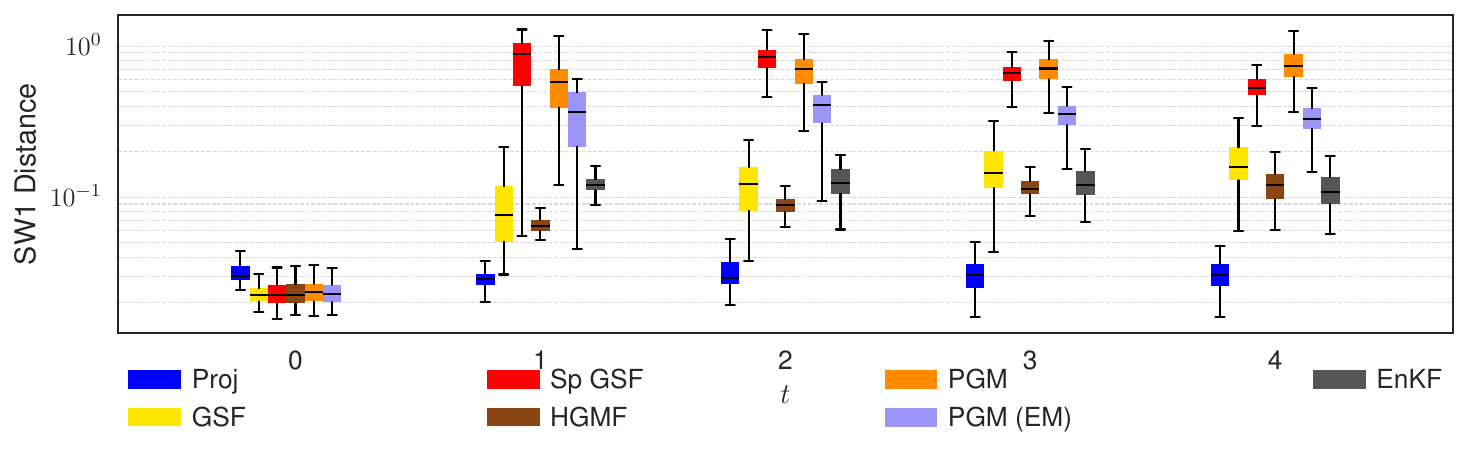}
  \caption{Quartile plots of the $\mathrm{SW}_1$ distances from the approximated posterior densities to the empirical posterior density for different methods. Although initially projection filter has a higher $\mathrm{SW}_1$ distance compared to the competing methods, the projection filter maintains low growth of $\mathrm{SW}_1$ distance over time samples, in contrast to the rest of the methods.}
  \label{fig:sw1_dist_comparison_fhn}
\end{figure}
\begin{figure}
  \centering
  \includegraphics[width=0.95\linewidth]{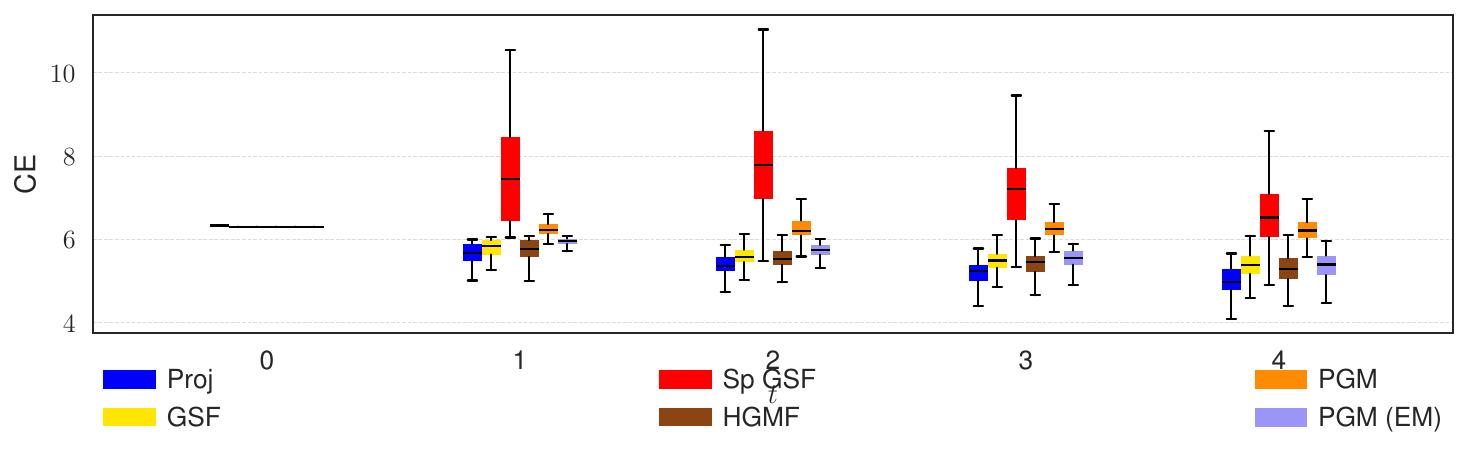}
  \caption{Cross entropy comparison between methods for FhN simulation, similar to Figure \ref{fig:cross_entropy_vdp}.}
  \label{fig:cross_entropy_fhn}
\end{figure}

\begin{figure}
  \centering
  \includegraphics[width=0.95\linewidth]{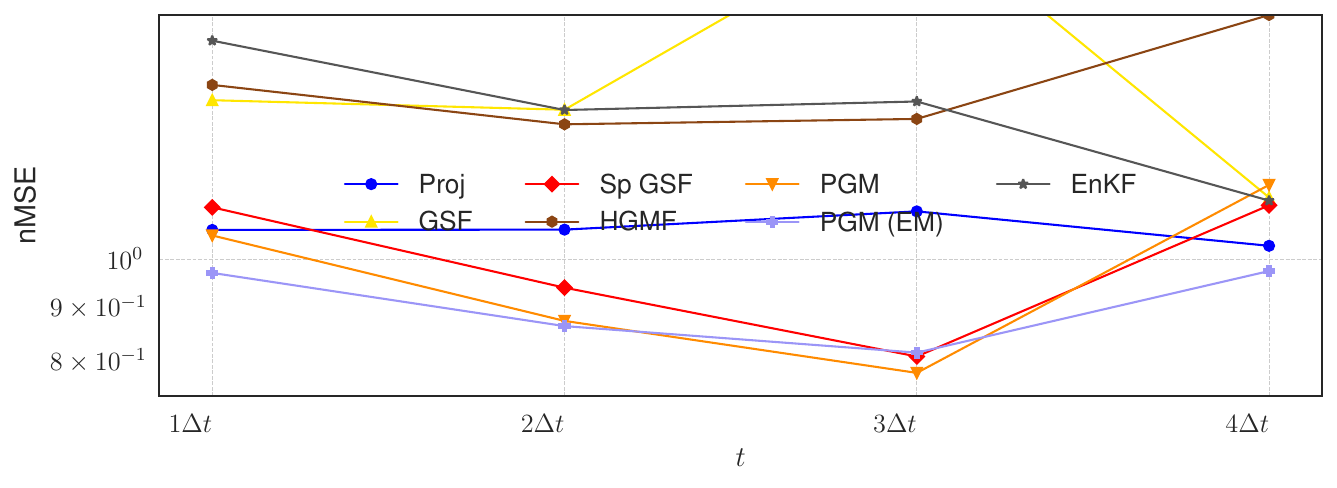}
  \caption{Comparison of nMSE for different methods for FhN simulation, similar to Figure \ref{fig:natural_statistics_mse}.}
  \label{fig:natural_statistics_mse_fhn}
\end{figure}
\section{Conclusions}\label{sec:Conclusions}

We have proposed a numerical implementation of the continuous-discrete projection filter for exponential-family manifolds where the densities within the manifolds are conjugate to the measurement likelihood. Specifically, by leveraging a combination of sparse-grid quadrature and adaptive bijection, the filter is derived for a class of continuous stochastic systems with discrete measurements subject to additive Gaussian noise. Due to the conjugate property, the filter benefits from an exact Bayesian update. We have also proposed a simple Riemannian gradient regularization to control the stiffness of the ODEs related to the prediction step of the natural parameters. The link between this regularization and stiffness is made precise in Proposition~\ref{prp:stiffness_bound_main}, and is reflected in the simulations by the higher completion rates of the regularized variants ($91\%\to 99\%$ for Van der Pol; $98\%$ for FhN). We have shown via a set of Monte Carlo simulations that the proposed filter significantly outperforms EnKF and Gaussian-sum filtering methods across several key metrics, while also being relatively less computationally demanding than PGM, PGM (EM), and EnKF methods.

The accuracy gains come at a higher per-step cost, but with far fewer propagated parameters than the Gaussian-mixture baselines. We are currently working on a sum-of-squares relaxation that constrains the natural parameter $\theta$ to remain inside $\Theta$ throughout the prediction step; initial results on this direction can be found in \cite{emzir2026sos}.

\section*{Funding}
M.F.E's work was supported in part by KFUPM-DROC, under Grant ER242009 and  the IRC Smart Mobility and Logistic KFUPM under Grant INML2407 and INML2528.
\section*{Disclosure statement}
The authors report there are no competing interests to declare.

\section*{Data availability statement}
No external datasets were used in this study. All figures and tables reported in this paper can be reproduced using the code archived on Zenodo at DOI \texttt{10.5281/zenodo.20177877}, which corresponds to the tagged release \texttt{v1.0} of the GitHub repository \url{https://github.com/puat133/CDC_Projfilter}.

\bibliographystyle{tfnlm}
\bibliography{ref}

\appendix

\section{Proof of Lemma \ref{lem:norm_zero_affine}}\label{app:proof_lemma}
\begin{proof}
If $c_2$ is an affine function of $c_1$, then there exist $N \in \mathbb{R}^{m_2 \times m_1}$ and $b \in \mathbb{R}^{m_2}$ such that $c_2 = N c_1 + b$. Evaluating $\E[\mathcal{L}(c_2)]$ directly shows that if $c_2$ is affine, then \eqref{eq:norm_zero_requirement} is satisfied. Let \eqref{eq:norm_zero_requirement} be satisfied, then, with $M(\theta) \coloneqq g_{(12)}(\theta)^\top g_{(11)}(\theta)^{-1}$, we have
$
    \E_{\theta} [\mathcal{L}(c_2 - M(\theta) c_1)] = 0.
$
Due to Assumption \ref{asm:x_and_x_square_in_c1}, $\mathcal{L}(c_2 - M(\theta) c_1)$ will be continuous. If $\mathcal{L}(c_2 - M(\theta) c_1)$ is non-zero in some subset of $\mathcal{S} \subset \mathbb{R}^d$, then we can take $\mathcal{B}(x_s, \delta) \in \mathcal{S}$, an open ball with radius $\delta$ centered on $x_s$. Then we can choose $p_{\theta}$ to be a Gaussian density with mean at $x_s$ and a small enough covariance matrix such that $\E_{\theta} [\mathcal{L}(c_2 - M(\theta) c_1)]$ can be arbitrarily made close to $\mathcal{L}(c_2(x_s) - M(\theta)c_1(x_s)) \neq 0$, which leads to a contradiction. Therefore, $\mathcal{L}(c_2 - M(\theta) c_1) = 0$ on $\mathbb{R}^d$. This implies that $c_2 - M(\theta) c_1$ belongs to the kernel of $\mathcal{L}$. Let us pick $c_2 - M(\theta)c_1 = \epsilon$, where $\epsilon \in \text{Ker}(\mathcal{L})$. Multiplying both sides by $\tilde{c}_1^\top$ and taking expectations w.r.t.\ $p_{\theta}$,
$
    \E_{\theta}[c_2 \tilde{c}_1^\top] - \E_{\theta}[M(\theta) c_1 \tilde{c}_1^\top] = \E_{\theta}[\epsilon \tilde{c}_1^\top]
$,
$
    g_{(12)}^\top(\theta) - M(\theta)g_{(11)}(\theta) = \E_{\theta}[(\epsilon - \E_{\theta}[\epsilon]) \tilde{c}_1^\top]
$
which means
$0 = \E_{\theta} [\tilde{c}_1 (\epsilon - \E_{\theta} [\epsilon])^\top] g_{(11)}(\theta)^{-1}$.
Hence $\E_{\theta} [\tilde{c}_1 (\epsilon - \E_{\theta} [\epsilon])^\top] = 0$, which means that $\epsilon$ is a constant function. Hence $c_2 = M(\theta)c_1 + b$ for some constant $b \in \mathbb{R}^{m_2}$. Since $M(\theta) = N$ for any $N \in \mathbb{R}^{m_2 \times m_1}$ such that $c_2 = N c_1 + b$, we have shown that $c_2$ is an affine function of $c_1$.
\end{proof}

\section{Proof of Proposition on Inequality \eqref{eq:strict_L_g_inv_L_inequality}}\label{app:proof_schur}
\begin{proof}
  Using $L_i \coloneqq \E_{\theta} [\mathcal{L}(c_i)]$, via the Schur complement, we obtain
\begin{align}
    \E_{\theta}[\mathcal{L}(c)]^\top g(\theta)^{-1}\E_{\theta}[\mathcal{L}(c)] =& L_1^\top g_{(11)} (\theta)^{-1}L_1 + \norm{\begin{bmatrix} -\sqrt{D(\theta)} g_{(12)}(\theta)^\top g_{(11)} (\theta)^{-1}  & \sqrt{D(\theta)} \end{bmatrix} \begin{bmatrix} L_1 \\ L_2 \end{bmatrix}}_2^2 \nonumber \\
    \geq& L_1^\top g_{(11)} (\theta)^{-1}L_1.\label{eq:L_g_inv_L_inequality_app}
\end{align}
where $D(\theta) \coloneqq \left( g(\theta) / g_{(11)}(\theta) \right)^{-1}$ is the inverse of the Schur complement of $g_{(11)}(\theta)$ from $g(\theta)$, and $\sqrt{D(\theta)}^\top \sqrt{D(\theta)} = D(\theta)$. The equality case of \eqref{eq:L_g_inv_L_inequality_app} occurs if
\begin{equation*}
    \sqrt{D(\theta)} g_{(12)}(\theta)^\top g_{(11)} (\theta)^{-1} \E_{\theta} [\mathcal{L}(c_1)] - \sqrt{D(\theta)} \E_{\theta} [\mathcal{L}(c_2)] = 0.
\end{equation*}
Since $D(\theta)$ is invertible, the equality can only happen if \eqref{eq:norm_zero_requirement} is satisfied.
Using Lemma \ref{lem:norm_zero_affine}, $c_2$ will only satisfy the equality condition if it is an affine function of $c_1$. However, since $g(\theta)$ is invertible and $c_2$ is linearly independent of $c_1$ and constants, the inequality \eqref{eq:strict_L_g_inv_L_inequality} is satisfied.
\end{proof}

\section{Proof of Proposition \ref{prp:conjugate_condition}}\label{app:proof_conjugate}
\begin{proof}
    Under the condition in the hypothesis, we have
    $
        -\ell_y = (y_k^\top R^{-1} H^\top) c_h - \frac{1}{2} (c_h^\top H R^{-1} H^\top c_h) - \frac{1}{2} y_k^\top R^{-1} y_k + K_\ell.
    $
    The terms $\frac{1}{2} y_k^\top R^{-1} y_k$ and $K_\ell$ will be canceled by normalization. Also,
    $
      -\frac{1}{2} c_h^\top H R^{-1} H^\top c_h = -\frac{1}{2} \text{tr}((H R^{-1} H^\top)(c_h c_h^\top))
        = -\frac{1}{2} \text{vec}(H R^{-1} H^\top)^\top (c_h \otimes c_h).
    $
    Hence, $\theta_\ell$ can be found by a linear combination and reindexing of $(y_k^\top R^{-1} H^\top)^\top$ and $-\frac{1}{2} \text{vec}(H R^{-1} H^\top)$. The condition that $c_h$ and $c_h \otimes c_h$ are in the span of $c$ can be written as \eqref{eq:T_1_T_2}. Using this, we finally have $\theta_\ell$ as required:
    $
        \theta_\ell = -T_1^\top H R^{-1} y_k + \frac{1}{2} T_2^\top \text{vec}(H R^{-1} H^\top).
    $
\end{proof}

\section{Proof of Proposition \ref{prp:stiffness_bound_main}}\label{app:proof_stiffness}
\begin{proof}
  Using $v(\theta) = \E_\theta^N[\mathcal{L}(c)]$, the Jacobian of $\omega_\epsilon$ with respect to $\theta_k$ is given by
  $
  \frac{\partial \omega_\epsilon}{\partial \theta_k} = G_\epsilon \frac{\partial v}{\partial \theta_k} + \frac{\partial G_\epsilon}{\partial \theta_k} v.
  $
  For the first term, observe that $\norm{G_\epsilon}_2 \leq 1/\epsilon$ by the definitions of $G_\epsilon(\theta)$ and $S_\epsilon$ in Algorithm \ref{alg:riem_grad}. Further, using the exponential family identity $\tfrac{\partial}{\partial \theta_k} \E_\theta[\varphi] = \mathrm{Cov}^N_\theta(\varphi,c_k)$, along with Hypothesis 1, we obtain
  \begin{equation}
  \left\|G_\epsilon \frac{\partial v}{\partial \theta_k}\right\|_2 \leq \|G_\epsilon\|_2 \|\mathrm{Cov}^N_\theta(\mathcal{L}(c), c_k)\|_2 \leq \frac{\bar{C}}{\epsilon}.\label{eq:bound_1_app}
  \end{equation}
  For the second term, differentiating $G_\epsilon(\theta)$ using product rule yields
  \begin{equation}
  \frac{\partial G_\epsilon}{\partial \theta_k} = V \left( \Omega_k S_\epsilon + \frac{\partial S_\epsilon}{\partial \theta_k} - S_\epsilon \Omega_k^\top \right) V^\top,\label{eq:dG_d_theta_second_term_app}
  \end{equation}
  where $\Omega_k \coloneqq V^\top \tfrac{\partial V}{\partial \theta_k}$ is skew symmetric since $V^\top V = I$.
  On the active indices where $\lambda_i > \epsilon$, we have $\tfrac{\partial (S_\epsilon)_{ii}}{\partial \theta_k} = -\tfrac{1}{\lambda_i^2} \tfrac{\partial \lambda_i}{\partial \theta_k}$.
  By Weyl's inequality \cite[Theorem 4.3.1]{horn2012}, $|\frac{\partial \lambda_i}{\partial \theta_k}| \leq \|\frac{\partial g^N}{\partial \theta_k}\|_2$, hence
  \begin{equation}
  \left\|\frac{\partial S_\epsilon}{\partial \theta_k}\right\|_2 \leq \frac{1}{\epsilon^2} \left\|\frac{\partial g^N}{\partial \theta_k}\right\|_2 \leq \frac{\bar{T}}{\epsilon^2}.\label{eq:bound_2_1_app}
  \end{equation}

  For the commutator term in \eqref{eq:dG_d_theta_second_term_app}, using Davis--Kahan theorem \cite{davis1970}, we have $\norm{\Omega_k}_2 \leq \frac{1}{\delta} \norm{\tfrac{\partial g^N(\theta)}{\partial \theta_k}}_2$. Therefore, since $\norm{S_\epsilon}_2 \leq 1/\epsilon$,
  \begin{equation}
  \|\Omega_k S_\epsilon - S_\epsilon \Omega_k^\top\|_2 \leq 2\|\Omega_k\|_2 \|S_\epsilon\|_2 \leq \frac{2\bar{T}}{\delta \epsilon}.\label{eq:bound_2_2_app}
  \end{equation}

Combining \eqref{eq:bound_1_app}, \eqref{eq:bound_2_1_app}, \eqref{eq:bound_2_2_app} and using Hypothesis 3, yields
$
  \left\|\frac{\partial G_\epsilon}{\partial \theta_k} v\right\|_2 \leq \left( \frac{2}{\delta} + \frac{1}{\epsilon} \right) \frac{\bar{T}\bar{V}}{\epsilon}.
$

Since $\frac{\partial \omega_\epsilon}{\partial \theta}$ is an $m \times m$ matrix with columns $\frac{\partial \omega_\epsilon}{\partial \theta_k}$, its spectral norm satisfies \eqref{eq:df_epsilon_d_theta_bound}.
\end{proof}

\end{document}

%% file: vdp_auto_generated_table.tex
\begin{table}
  \centering
  \footnotesize
  \begin{tabular}{|l|c|c|}
  \hline
  \textbf{Method} & \textbf{Execution Times (s) ($\downarrow$)} & \textbf{FLOPS ($\downarrow$)} \\ \hline
  Particle & $2.249 \times 10^{1} \pm 3.954 \times 10^{-1}$ & $1.901 \times 10^{10} \pm 2.968 \times 10^{-7}$ \\ \hline
  Proj & $8.987 \times 10^{-1} \pm 1.802 \times 10^{-1}$ & $8.408 \times 10^{8} \pm 1.095 \times 10^{8}$ \\ \hline
  Proj-0 & $9.150 \times 10^{-1} \pm 1.953 \times 10^{-1}$ & $8.408 \times 10^{8} \pm 1.095 \times 10^{8}$ \\ \hline
  Proj-B & $9.428 \times 10^{-1} \pm 2.061 \times 10^{-1}$ & $8.419 \times 10^{8} \pm 1.103 \times 10^{8}$ \\ \hline
  GSF & $1.357 \times 10^{-3} \pm 8.633 \times 10^{-5}$ & $1.427 \times 10^{5} \pm 8.295 \times 10^{3}$ \\ \hline
  SP-GSF & $4.151 \times 10^{-3} \pm 6.585 \times 10^{-4}$ & $2.317 \times 10^{6} \pm 4.349 \times 10^{5}$ \\ \hline
  HGMF & $4.689 \times 10^{-3} \pm 9.110 \times 10^{-4}$ & $1.774 \times 10^{6} \pm 4.957 \times 10^{5}$ \\ \hline
  PGM (K-means) & $1.248 \times 10^{1} \pm 2.451 \times 10^{-1}$ & $4.699 \times 10^{9} \pm 7.336 \times 10^{-8}$ \\ \hline
  PGM (EM) & $6.082 \times 10^{0} \pm 2.905 \times 10^{-2}$ & $8.745 \times 10^{9} \pm 1.365 \times 10^{-7}$ \\ \hline
  EnKF & $1.433 \times 10^{1} \pm 2.041 \times 10^{-1}$ & $2.112 \times 10^{10} \pm 3.297 \times 10^{-7}$ \\ \hline
  \end{tabular}
  \caption{Execution times and FLOPS for Van der Pol example.}
  \label{tab:result-vdp}
\end{table}

%% file: fhn_auto_generated_table.tex
\begin{table}
  \centering
  \footnotesize
  \begin{tabular}{|l|c|c|}
  \hline
  \textbf{Method} & \textbf{Execution Times (s) ($\downarrow$)} & \textbf{FLOPS ($\downarrow$)} \\ \hline
  Particle & $5.716 \times 10^{1} \pm 1.980 \times 10^{0}$ & $6.720 \times 10^{10} \pm 0.000 \times 10^{0}$ \\ \hline
  Projection & $1.162 \times 10^{1} \pm 6.470 \times 10^{0}$ & $3.591 \times 10^{10} \pm 2.106 \times 10^{10}$ \\ \hline
  GSF & $3.216 \times 10^{-3} \pm 1.933 \times 10^{-4}$ & $6.214 \times 10^{5} \pm 1.183 \times 10^{4}$ \\ \hline
  SP-GSF & $5.203 \times 10^{-2} \pm 1.159 \times 10^{-2}$ & $1.441 \times 10^{8} \pm 2.718 \times 10^{7}$ \\ \hline
  PGM (K-means) & $1.136 \times 10^{2} \pm 3.216 \times 10^{0}$ & $9.633 \times 10^{10} \pm 0.000 \times 10^{0}$ \\ \hline
  PGM (EM) & $9.096 \times 10^{1} \pm 2.007 \times 10^{0}$ & $1.578 \times 10^{11} \pm 0.000 \times 10^{0}$ \\ \hline
  EnKF & $4.762 \times 10^{1} \pm 1.570 \times 10^{0}$ & $7.469 \times 10^{10} \pm 0.000 \times 10^{0}$ \\ \hline
  HGMF & $5.191 \times 10^{-2} \pm 6.376 \times 10^{-3}$ & $1.184 \times 10^{8} \pm 1.528 \times 10^{7}$ \\ \hline
  \end{tabular}
  \caption{Execution times and FLOPS for FitzHugh--Nagumo example.}
  \label{tab:result-fhn}
\end{table}

%% file: ref.bib
@inproceedings{hanzon1991,
  title = {New Results on the Projection Filter},
  booktitle = {European {{Control Conference}}},
  author = {Hanzon, B and Hut, R},
  year = 1991,
  month = jul,
  pages = {9},
  address = {Grenoble},
  langid = {english}
}

@techreport{brigo1995,
  type = {Research Report},
  title = {A Differential Geometric Approach to Nonlinear Filtering: The Projection Filter},
  author = {Brigo, Damiano and Hanzon, Bernard and Gland, Fran{\c c}ois Le},
  year = 1995,
  number = {2598},
  institution = {INRIA Rennes - Bretagne Atlantique}
}

@article{brigo1998,
  title = {A Differential Geometric Approach to Nonlinear Filtering: The Projection Filter},
  author = {Brigo, Damiano and Hanzon, Bernard and Gland, Fran{\c c}ois Le},
  year = 1998,
  journal = {IEEE Transactions on Automatic Control},
  volume = {43},
  number = {2},
  pages = {247--252},
  publisher = {{Institute of Electrical and Electronics Engineers (IEEE)}},
  doi = {10.1109/9.661075}
}

@article{brigo1999,
  title = {Approximate Nonlinear Filtering by Projection on Exponential Manifolds of Densities},
  author = {Brigo, Damiano and Hanzon, Bernard and Gland, Fran{\c c}ois Le},
  year = 1999,
  month = jun,
  journal = {Bernoulli},
  volume = {5},
  number = {3},
  pages = {495},
  publisher = {JSTOR},
  doi = {10.2307/3318714}
}

@article{armstrong2016,
  title = {Nonlinear Filtering via Stochastic {{PDE}} Projection on Mixture Manifolds in {{L}}{$^2$} Direct Metric},
  author = {Armstrong, John and Brigo, Damiano},
  year = 2016,
  month = dec,
  journal = {Mathematics of Control, Signals, and Systems},
  volume = {28},
  number = {1},
  pages = {5},
  publisher = {{Springer Science and Business Media LLC}},
  doi = {10.1007/s00498-015-0154-1}
}

@article{armstrong2016a,
  title = {Optimal Approximation of {{SDEs}} on Submanifolds: The {{It\^o-vector}} and {{It\^o-jet}} Projections},
  author = {Armstrong, John and Brigo, Damiano and {Rossi Ferrucci}, Emilio},
  year = 2019,
  month = jul,
  journal = {Proceedings of the London Mathematical Society},
  volume = {119},
  number = {1},
  pages = {176--213},
  doi = {10.1112/plms.12226},
  eprint = {1610.03887},
  primaryclass = {math.PR},
  archiveprefix = {arXiv}
}

@article{armstrong2023,
  title = {Optimal Projection Filters with Information Geometry},
  author = {Armstrong, John and Brigo, Damiano and Hanzon, Bernard},
  year = 2023,
  month = jun,
  journal = {Info. Geo.},
  issn = {2511-249X},
  doi = {10.1007/s41884-023-00108-x},
  urldate = {2023-06-13},
  langid = {english}
}

@article{gao2020a,
  title = {Design of a {{Quantum Projection Filter}}},
  author = {Gao, Qing and Dong, Daoyi and Petersen, Ian R. and Ding, Steven X.},
  year = 2020,
  month = aug,
  journal = {IEEE Transactions on Automatic Control},
  volume = {65},
  number = {8},
  pages = {3693--3700},
  issn = {1558-2523},
  doi = {10.1109/TAC.2019.2953457}
}

@article{kutschireiter2022,
  title = {Projection {{Filtering}} with {{Observed State Increments}} with {{Applications}} in {{Continuous-Time Circular Filtering}}},
  author = {Kutschireiter, Anna and Rast, Luke and Drugowitsch, Jan},
  year = 2022,
  journal = {IEEE Trans. Signal Process.},
  volume = {70},
  eprint = {2102.09650},
  primaryclass = {cs, stat},
  pages = {686--700},
  issn = {1053-587X, 1941-0476},
  doi = {10.1109/TSP.2022.3143471},
  urldate = {2022-09-03},
  archiveprefix = {arXiv}
}

@article{emzir2023,
  title = {Gaussian-{{Based Parametric Bijections For Automatic Projection Filters}}},
  author = {Emzir, Muhammad F. and Zhao, Zheng and Cheded, Lahouari and S{\"a}rkk{\"a}, Simo},
  year = 2023,
  month = dec,
  journal = {IEEE Trans. Automat. Contr.},
  pages = {1--8},
  doi = {10.1109/TAC.2023.3340979},
  urldate = {2023-04-02}
}

@article{emzir2023a,
  title = {Multidimensional Projection Filters via Automatic Differentiation and Sparse-Grid Integration},
  author = {Emzir, Muhammad Fuady and Zhao, Zheng and S{\"a}rkk{\"a}, Simo},
  year = 2023,
  month = mar,
  journal = {Signal Processing},
  volume = {204},
  pages = {108832},
  issn = {01651684},
  doi = {10.1016/j.sigpro.2022.108832},
  urldate = {2022-11-09},
  langid = {english}
}

@inproceedings{emzir2024,
  title = {A Projection Filter Algorithm for Stochastic Systems with Correlated Noise and State-Dependent Measurement Covariance},
  booktitle = {2024 {{American Control Conference}} ({{ACC}})},
  author = {Emzir, Muhammad and Cheded, Lahouari},
  year = 2024,
  pages = {1909--1914},
  address = {Toronto},
  doi = {10.23919/ACC60939.2024.10644170}
}

@article{emzir2024a,
  title = {Efficient Projection Filter Algorithm for Stochastic Dynamical Systems with Correlated Noises and State-Dependent Measurement Covariance},
  author = {Emzir, Muhammad Fuady},
  year = 2024,
  month = may,
  journal = {Signal Processing},
  volume = {218},
  pages = {109383},
  issn = {0165-1684},
  doi = {10.1016/j.sigpro.2024.109383},
  urldate = {2024-05-10}
}

@article{emzir2024b,
  title = {It\^o-Vector Projection Filter for Exponential Families},
  author = {Emzir, Muhammad F.},
  year = 2024,
  journal = {Results Appl. Math.},
  doi = {10.1016/j.rinam.2024.100492}
}

@book{jazwinski1970,
  title = {Stochastic Processes and Filtering Theory},
  author = {Jazwinski, Andrew},
  year = 1970,
  publisher = {Academic Press},
  address = {New York},
  isbn = {978-0-08-096090-6}
}

@article{kulikov2019,
  title = {Numerical Robustness of Extended {{Kalman}} Filtering Based State Estimation in Ill-Conditioned Continuous-Discrete Nonlinear Stochastic Chemical Systems},
  author = {Kulikov, G. {\relax Yu}. and Kulikova, M. V.},
  year = 2019,
  journal = {International Journal of Robust and Nonlinear Control},
  volume = {29},
  number = {5},
  pages = {1377--1395},
  issn = {1099-1239},
  doi = {10.1002/rnc.4440},
  urldate = {2026-02-10},
  copyright = {\copyright{} 2018 John Wiley \& Sons, Ltd.},
  langid = {english}
}

@article{kulikov2022,
  title = {Universal {{MATLAB}}-based Square-root Solutions in the Family of Continuous-discrete {{Gaussian}} Filters for State Estimation in Nonlinear Stochastic Dynamic Systems},
  author = {Kulikov, Gennady Yurievich and Kulikova, Maria Vyacheslavovna},
  year = 2022,
  month = oct,
  journal = {Intl J Robust \& Nonlinear},
  volume = {32},
  number = {15},
  pages = {8227--8251},
  issn = {1049-8923, 1099-1239},
  doi = {10.1002/rnc.6268},
  urldate = {2026-02-10},
  langid = {english}
}

@article{wang2022e,
  title = {The {{Level Set Kalman Filter}} for {{State Estimation}} of {{Continuous-Discrete Systems}}},
  author = {Wang, Ningyuan and Forger, Daniel B.},
  year = 2022,
  journal = {IEEE Transactions on Signal Processing},
  volume = {70},
  pages = {631--642},
  issn = {1941-0476},
  doi = {10.1109/TSP.2021.3133698},
  urldate = {2024-08-22}
}

@inproceedings{li2018,
  title = {Robust {{Nonlinear Filter Using Adaptive Edgeworth Expansion}}},
  booktitle = {2018 {{Annual American Control Conference}} ({{ACC}})},
  author = {Li, Dawei and Xin, Ming and Jia, Bin},
  year = 2018,
  month = jun,
  pages = {1927--1932},
  issn = {2378-5861},
  doi = {10.23919/ACC.2018.8430825}
}

@article{chen2021,
  title = {Solving {{Inverse Stochastic Problems}} from {{Discrete Particle Observations Using}} the {{Fokker--Planck Equation}} and {{Physics-Informed Neural Networks}}},
  author = {Chen, Xiaoli and Yang, Liu and Duan, Jinqiao and Karniadakis, George Em},
  year = 2021,
  month = jan,
  journal = {SIAM J. Sci. Comput.},
  volume = {43},
  number = {3},
  pages = {B811-B830},
  publisher = {{Society for Industrial and Applied Mathematics}},
  issn = {1064-8275},
  doi = {10.1137/20M1360153},
  urldate = {2023-11-25}
}

@inproceedings{daum1984,
  title = {Exact Finite Dimensional Nonlinear Filters for Continuous Time Processes with Discrete Time Measurements},
  booktitle = {The 23rd {{IEEE Conference}} on {{Decision}} and {{Control}}},
  author = {Daum, Frederick E.},
  year = 1984,
  month = dec,
  pages = {16--22},
  doi = {10.1109/CDC.1984.272243},
  urldate = {2024-08-18}
}

@book{kass1997,
  title = {Geometrical Foundations of Asymptotic Inference},
  author = {Kass, Robert E. and Vos, Paul W.},
  year = 1997,
  series = {Wiley Series in Probability and Statistics},
  publisher = {Wiley},
  address = {New York},
  isbn = {978-0-471-82668-2},
  lccn = {QA276 .K228 1997}
}

@book{amari2000,
  title = {Methods of Information Geometry},
  author = {Amari, Shun'ichi and Nagaoka, Hiroshi},
  year = 2000,
  series = {Translations of Mathematical Monographs},
  number = {v. 191},
  publisher = {American Mathematical Society},
  address = {Providence, RI},
  isbn = {978-0-8218-0531-2},
  langid = {english},
  lccn = {QA276 .A56313 2000}
}

@misc{brigo2009,
  title = {Projecting the {{Fokker-Planck Equation}} onto a Finite Dimensional Exponential Family},
  author = {Brigo, Damiano and Pistone, Giovanni},
  year = 2009,
  month = jan,
  number = {arXiv:0901.1308},
  eprint = {0901.1308},
  primaryclass = {math},
  publisher = {arXiv},
  doi = {10.48550/arXiv.0901.1308},
  urldate = {2024-08-22},
  archiveprefix = {arXiv}
}

@article{smolyak1963,
  title = {Quadrature and Interpolation Formulas for Tensor Products of Certain Classes of Functions},
  author = {Smolyak, S. A.},
  year = 1963,
  journal = {Dokl. Akad. Nauk SSSR},
  volume = {148},
  pages = {1042--1045}
}

@article{novak1996,
  title = {High Dimensional Integration of Smooth Functions over Cubes},
  author = {Novak, Erich and Ritter, Klaus},
  year = 1996,
  month = nov,
  journal = {Numerische Mathematik},
  volume = {75},
  number = {1},
  pages = {79--97},
  publisher = {{Springer Science and Business Media LLC}},
  doi = {10.1007/s002110050231}
}

@article{gerstner1998,
  title = {Numerical Integration Using Sparse Grids},
  author = {Gerstner, Thomas and Griebel, Michael},
  year = 1998,
  journal = {Numerical Algorithms},
  volume = {18},
  number = {3/4},
  pages = {209--232},
  publisher = {{Springer Science and Business Media LLC}},
  doi = {10.1023/a:1019129717644}
}

@book{horn2012,
  title = {Matrix Analysis},
  author = {Horn, Roger A. and Johnson, Charles R.},
  year = 2012,
  edition = {2nd ed},
  publisher = {Cambridge University Press},
  address = {Cambridge ; New York},
  isbn = {978-0-521-83940-2},
  lccn = {QA188 .H66 2012}
}

@article{davis1970,
  title = {The {{Rotation}} of {{Eigenvectors}} by a {{Perturbation}}. {{III}}},
  author = {Davis, Chandler and Kahan, W. M.},
  year = 1970,
  journal = {SIAM Journal on Numerical Analysis},
  publisher = {{Society for Industrial and Applied Mathematics}},
  doi = {10.1137/0707001},
  urldate = {2026-01-20},
  copyright = {Copyright \copyright{} 1970 Society for Industrial and Applied Mathematics},
  langid = {english}
}

@article{martens2020,
  title = {New {{Insights}} and {{Perspectives}} on the {{Natural Gradient Method}}},
  author = {Martens, James},
  year = 2020,
  journal = {Journal of Machine Learning Research},
  volume = {21},
  number = {146},
  pages = {1--76},
  issn = {1533-7928},
  urldate = {2026-01-18}
}

@incollection{schmudgen2017,
  title = {The {{Moment Problem}} on {{Compact Semi-Algebraic Sets}}},
  booktitle = {The {{Moment Problem}}},
  author = {Schm{\"u}dgen, Konrad},
  editor = {Schm{\"u}dgen, Konrad},
  year = 2017,
  pages = {283--313},
  publisher = {Springer International Publishing},
  address = {Cham},
  doi = {10.1007/978-3-319-64546-9_12},
  urldate = {2025-02-04},
  isbn = {978-3-319-64546-9},
  langid = {english}
}

@article{tsitouras2011,
  title = {Runge--{{Kutta}} Pairs of Order 5(4) Satisfying Only the First Column Simplifying Assumption},
  author = {Tsitouras, {\relax Ch}.},
  year = 2011,
  month = jul,
  journal = {Computers \& Mathematics with Applications},
  volume = {62},
  number = {2},
  pages = {770--775},
  issn = {0898-1221},
  doi = {10.1016/j.camwa.2011.06.002},
  urldate = {2024-05-09}
}

@article{quininao2020,
  title = {Clamping and {{Synchronization}} in the {{Strongly Coupled FitzHugh--Nagumo Model}}},
  author = {Qui{\textbackslash}{\textasciitilde}{\textbraceleft}n{\textbraceright}inao, Cristobal and Touboul, Jonathan D.},
  year = 2020,
  month = jan,
  journal = {SIAM J. Appl. Dyn. Syst.},
  volume = {19},
  number = {2},
  pages = {788--827},
  publisher = {{Society for Industrial and Applied Mathematics}},
  doi = {10.1137/19M1283884},
  urldate = {2025-12-17}
}

@article{voss2004,
  title = {Nonlinear Dynamical System Identification from Uncertain and Indirect Measurements},
  author = {Voss, Henning U. and Timmer, Jens and Kurths, J{\"u}rgen},
  year = 2004,
  month = jun,
  journal = {Int. J. Bifurcation Chaos},
  volume = {14},
  number = {06},
  pages = {1905--1933},
  publisher = {World Scientific Publishing Co.},
  issn = {0218-1274},
  doi = {10.1142/S0218127404010345},
  urldate = {2025-12-17}
}

@book{smith2019,
  title = {Cellular Biophysics and Modeling: A Primer on the Computational Biology of Excitable Cells},
  shorttitle = {Cellular Biophysics and Modeling},
  author = {Smith, Greg Conradi},
  year = 2019,
  publisher = {Cambridge university press},
  address = {Cambridge, United Kingdom},
  isbn = {978-1-107-00536-5 978-0-521-18305-5},
  langid = {english},
  lccn = {612.014}
}

@article{fitzhugh1955,
  title = {Mathematical Models of Threshold Phenomena in the Nerve Membrane},
  author = {FitzHugh, Richard},
  year = 1955,
  month = dec,
  journal = {Bulletin of Mathematical Biophysics},
  volume = {17},
  number = {4},
  pages = {257--278},
  issn = {1522-9602},
  doi = {10.1007/BF02477753},
  urldate = {2025-12-24},
  langid = {english}
}

@article{nagumo1962,
  title = {An {{Active Pulse Transmission Line Simulating Nerve Axon}}},
  author = {Nagumo, J. and Arimoto, S. and Yoshizawa, S.},
  year = 1962,
  month = oct,
  journal = {Proceedings of the IRE},
  volume = {50},
  number = {10},
  pages = {2061--2070},
  issn = {2162-6634},
  doi = {10.1109/JRPROC.1962.288235},
  urldate = {2025-12-24}
}

@article{bonneel2015,
  title = {Sliced and {{Radon Wasserstein Barycenters}} of {{Measures}}},
  author = {Bonneel, Nicolas and Rabin, Julien and Peyr{\'e}, Gabriel and Pfister, Hanspeter},
  year = 2015,
  month = jan,
  journal = {J Math Imaging Vis},
  volume = {51},
  number = {1},
  pages = {22--45},
  issn = {1573-7683},
  doi = {10.1007/s10851-014-0506-3},
  urldate = {2025-12-16},
  langid = {english}
}

@inproceedings{kolouri2019,
  title = {Generalized {{Sliced Wasserstein Distances}}},
  booktitle = {Advances in {{Neural Information Processing Systems}}},
  author = {Kolouri, Soheil and Nadjahi, Kimia and Simsekli, Umut and Badeau, Roland and Rohde, Gustavo},
  year = 2019,
  volume = {32},
  pages = {261--272},
  publisher = {Curran Associates, Inc.},
  urldate = {2025-12-16}
}

@book{chopin2020,
  title = {An Introduction to Sequential {{Monte}}--{{Carlo}}},
  author = {Chopin, Nicolas},
  year = 2020,
  publisher = {Springer},
  address = {Cham, Switzerland},
  isbn = {978-3-030-47845-2}
}

@article{bishop2023,
  title = {On the Mathematical Theory of Ensemble (Linear-{{Gaussian}}) {{Kalman}}--{{Bucy}} Filtering},
  author = {Bishop, Adrian N. and Del Moral, Pierre},
  year = 2023,
  month = dec,
  journal = {Math. Control Signals Syst.},
  volume = {35},
  number = {4},
  pages = {835--903},
  issn = {0932-4194, 1435-568X},
  doi = {10.1007/s00498-023-00357-2},
  urldate = {2024-09-28},
  langid = {english}
}

@misc{kidger2022,
  title = {On {{Neural Differential Equations}}},
  author = {Kidger, Patrick},
  year = 2022,
  month = feb,
  number = {arXiv:2202.02435},
  eprint = {2202.02435},
  primaryclass = {cs, math, stat},
  publisher = {arXiv},
  doi = {10.48550/arXiv.2202.02435},
  urldate = {2024-05-09},
  archiveprefix = {arXiv}
}

@inproceedings{emzir2026sos,
  title = {On Solving Continuous-Discrete Projection Filters via Sum-of-Squares Relaxation},
  author = {Emzir, Muhammad Fuady},
  year = 2026,
  booktitle = {Asian Control Conference (ASCC)},
  pages = {(to appear)},
  note = {Accepted.},
}

@article{bungartz2004,
  author = {Bungartz, Hans-Joachim and Griebel, Michael},
  title = {Sparse Grids},
  journal = {Acta Numerica},
  volume = {13},
  pages = {147--269},
  year = {2004},
  publisher = {Cambridge University Press},
  doi = {10.1017/S0962492904000182},
  url = {https://doi.org/10.1017/S0962492904000182},
}

@book{boumal2023,
  author = {Boumal, Nicolas},
  title = {An Introduction to Optimization on Smooth Manifolds},
  publisher = {Cambridge University Press},
  year = {2023},
  doi = {10.1017/9781009166164},
  url = {https://doi.org/10.1017/9781009166164},
}

@book{sarkka2013,
  author = {S{\"a}rkk{\"a}, Simo},
  title = {Bayesian Filtering and Smoothing},
  publisher = {Cambridge University Press},
  year = {2013},
  series = {Institute of Mathematical Statistics Textbooks},
  volume = {3},
  doi = {10.1017/CBO9781139344203},
  url = {https://doi.org/10.1017/CBO9781139344203},
}

@article{emzir2025a,
  title = {A {{Bayesian Update Method}} for {{Exponential Family Projection Filters}} with {{Non-Conjugate Likelihoods}}},
  author = {Emzir, Muhammad Fuady},
  year = 2025,
  journal = {Systems \& Control Letters},
  url = {https://www.sciencedirect.com/science/article/pii/S0167691125002166},
  note = {Preprint version available as arXiv:2504.16867, doi:10.48550/arXiv.2504.16867},
  urldate = {2025-04-24},
  archiveprefix = {arXiv}
}

@book{evensen2009,
  title = {Data {{Assimilation}}: {{The Ensemble Kalman Filter}}},
  shorttitle = {Data {{Assimilation}}},
  author = {Evensen, Geir},
  year = 2009,
  month = aug,
  publisher = {Springer Science \& Business Media},
  googlebooks = {2\_zaTb\_O1AkC},
  isbn = {978-3-642-03711-5},
  langid = {english}
}

@article{hu2025,
  title = {Score-{{Based Physics-Informed Neural Networks}} for {{High-Dimensional Fokker}}--{{Planck Equations}}},
  author = {Hu, Zheyuan and Zhang, Zhongqiang and Karniadakis, George E. and Kawaguchi, Kenji},
  year = 2025,
  month = jun,
  journal = {SIAM J. Sci. Comput.},
  volume = {47},
  number = {3},
  pages = {C680-C705},
  publisher = {{Society for Industrial and Applied Mathematics}},
  issn = {1064-8275},
  doi = {10.1137/24M1638768},
  urldate = {2025-12-03}
}

@article{dipaola2002,
  title = {Approximate Solution of the {{Fokker}}--{{Planck}}--{{Kolmogorov}} Equation},
  author = {Di Paola, M. and Sofi, A.},
  year = 2002,
  month = oct,
  journal = {Probabilistic Engineering Mechanics},
  volume = {17},
  number = {4},
  pages = {369--384},
  issn = {0266-8920},
  doi = {10.1016/S0266-8920(02)00034-6},
  urldate = {2025-11-25}
}

@book{risken1996,
  title = {The {{Fokker-Planck Equation}}: {{Methods}} of {{Solution}} and {{Applications}}},
  shorttitle = {The {{Fokker-Planck Equation}}},
  author = {Risken, Hannes},
  year = 1996,
  series = {Springer {{Series}} in {{Synergetics}}},
  volume = {18},
  publisher = {Springer},
  address = {Berlin, Heidelberg},
  doi = {10.1007/978-3-642-61544-3},
  urldate = {2025-11-25},
  copyright = {https://www.springernature.com/gp/researchers/text-and-data-mining},
  isbn = {978-3-540-61530-9 978-3-642-61544-3},
  langid = {english}
}

@article{beskos2017,
  title = {A Stable Particle Filter for a Class of High-Dimensional State-Space Models},
  author = {Beskos, Alexandros and Crisan, Dan and Jasra, Ajay and Kamatani, Kengo and Zhou, Yan},
  year = 2017,
  month = mar,
  journal = {Advances in Applied Probability},
  volume = {49},
  number = {1},
  pages = {24--48},
  publisher = {Cambridge University Press (CUP)},
  doi = {10.1017/apr.2016.77}
}

@article{snyder2008,
  title = {Obstacles to High-Dimensional Particle Filtering},
  author = {Snyder, Chris and Bengtsson, Thomas and Bickel, Peter and Anderson, Jeff},
  year = 2008,
  month = dec,
  journal = {Monthly Weather Review},
  volume = {136},
  number = {12},
  pages = {4629--4640},
  publisher = {American Meteorological Society},
  doi = {10.1175/2008mwr2529.1}
}

@article{xia2013,
  title = {A New Continuous-Discrete Particle Filter for Continuous-Discrete Nonlinear Systems},
  author = {Xia, Yuanqing and Deng, Zhihong and Li, Li and Geng, Xiumei},
  year = 2013,
  month = sep,
  journal = {Information Sciences},
  volume = {242},
  pages = {64--75},
  issn = {0020-0255},
  doi = {10.1016/j.ins.2013.04.030},
  urldate = {2024-08-22}
}

@inproceedings{sarkka2006,
  title = {On {{Sequential Monte Carlo Sampling}} of {{Discretely Observed Stochastic Differential Equations}}},
  booktitle = {2006 {{IEEE Nonlinear Statistical Signal Processing Workshop}}},
  author = {S{\"a}rkk{\"a}, Simo},
  year = 2006,
  month = sep,
  pages = {21--24},
  doi = {10.1109/NSSPW.2006.4378811},
  urldate = {2024-08-22}
}

@article{singer2008,
  title = {Generalized {{Gauss}}--{{Hermite}} Filtering},
  author = {Singer, Hermann},
  year = 2008,
  month = may,
  journal = {AStA},
  volume = {92},
  number = {2},
  pages = {179--195},
  issn = {1863-818X},
  doi = {10.1007/s10182-008-0068-z},
  urldate = {2023-03-23},
  langid = {english}
}

@article{arasaratnam2010,
  title = {Cubature {{Kalman Filtering}} for {{Continuous-Discrete Systems}}: {{Theory}} and {{Simulations}}},
  shorttitle = {Cubature {{Kalman Filtering}} for {{Continuous-Discrete Systems}}},
  author = {Arasaratnam, Ienkaran and Haykin, Simon and Hurd, Thomas R.},
  year = 2010,
  month = oct,
  journal = {IEEE Transactions on Signal Processing},
  volume = {58},
  number = {10},
  pages = {4977--4993},
  issn = {1941-0476},
  doi = {10.1109/TSP.2010.2056923},
  urldate = {2024-03-20}
}

@article{knudsen2019,
  title = {A {{New Continuous Discrete Unscented Kalman Filter}}},
  author = {Knudsen, Torben and Leth, John},
  year = 2019,
  month = may,
  journal = {IEEE Transactions on Automatic Control},
  volume = {64},
  number = {5},
  pages = {2198--2205},
  issn = {1558-2523},
  doi = {10.1109/TAC.2018.2867325},
  urldate = {2024-08-22}
}

@article{alspach1972,
  title = {Nonlinear {{Bayesian}} Estimation Using {{Gaussian}} Sum Approximations},
  author = {Alspach, D. and Sorenson, H.},
  year = 1972,
  month = aug,
  journal = {IEEE Transactions on Automatic Control},
  volume = {17},
  number = {4},
  pages = {439--448},
  issn = {1558-2523},
  doi = {10.1109/TAC.1972.1100034},
  urldate = {2024-05-09}
}

@article{terejanu2011,
  title = {Adaptive {{Gaussian Sum Filter}} for {{Nonlinear Bayesian Estimation}}},
  author = {Terejanu, Gabriel and Singla, Puneet and Singh, Tarunraj and Scott, Peter D.},
  year = 2011,
  month = sep,
  journal = {IEEE Trans. Automat. Contr.},
  volume = {56},
  number = {9},
  pages = {2151--2156},
  issn = {0018-9286, 1558-2523},
  doi = {10.1109/TAC.2011.2141550},
  urldate = {2024-09-15},
  copyright = {https://ieeexplore.ieee.org/Xplorehelp/downloads/license-information/IEEE.html}
}

@article{raihan2018,
  title = {Particle {{Gaussian}} Mixture Filters-{{I}}},
  author = {Raihan, Dilshad and Chakravorty, Suman},
  year = 2018,
  month = dec,
  journal = {Automatica},
  volume = {98},
  pages = {331--340},
  issn = {0005-1098},
  doi = {10.1016/j.automatica.2018.07.023},
  urldate = {2024-09-18}
}

@article{craft2025,
  title = {Homotopic {{Gaussian Mixture Filtering}} for {{Applied Bayesian Inference}}},
  author = {Craft, Kyle J. and DeMars, Kyle J.},
  year = 2025,
  month = jul,
  journal = {IEEE Transactions on Automatic Control},
  volume = {70},
  number = {7},
  pages = {4608--4623},
  issn = {1558-2523},
  doi = {10.1109/TAC.2025.3530878},
  urldate = {2026-01-10}
}

@article{kim2026,
  title = {{{DBSCAN-based}} Particle {{Gaussian}} Mixture Filters},
  author = {Kim, Sukkeun and Sun, Mengwei and Petrunin, Ivan and Shin, Hyo-Sang},
  year = 2026,
  month = jan,
  journal = {Digital Signal Processing},
  volume = {168},
  pages = {105546},
  issn = {1051-2004},
  doi = {10.1016/j.dsp.2025.105546},
  urldate = {2026-01-10}
}

@article{challa2000,
  title = {Nonlinear Filtering via Generalized {{Edgeworth}} Series and {{Gauss-Hermite}} Quadrature},
  author = {Challa, S. and {Bar-Shalom}, Y. and Krishnamurthy, V.},
  year = 2000,
  month = jun,
  journal = {IEEE Transactions on Signal Processing},
  volume = {48},
  number = {6},
  pages = {1816--1820},
  issn = {1941-0476},
  doi = {10.1109/78.845944}
}
